\documentclass[11pt]{article}
\usepackage{amsmath}
\usepackage{amsthm}
\usepackage{cite}
\usepackage{amscd}
\usepackage{authblk}
\usepackage[all]{xy}
\usepackage{titlesec}
\titleformat*{\section}{\bfseries\sffamily}
\titleformat*{\subsection}{\bfseries\sffamily}
\usepackage{secdot}
\makeatletter
\renewcommand{\title}[1]{\gdef\@title{\Large\bfseries\sffamily#1}}
\makeatother

\theoremstyle{definition}
\newtheorem{thm}{Theorem}[section]
\newtheorem{definition}[thm]{Definition}
\newtheorem{prop}[thm]{Proposition}
\newtheorem{lem}[thm]{Lemma}
\newtheorem{rem}[thm]{Remark}
\newtheorem{cor}[thm]{Corollary}
\numberwithin{equation}{section}
\usepackage{amssymb}
\usepackage[margin=30truemm]{geometry}
\newcommand{\Spec}{\operatorname{Spec}}
\newcommand{\rank}{{\rm rk}}
\newcommand{\Hom}{{\rm Hom}}
\newcommand{\Coh}{\rm Coh}
\newcommand{\Rep}{{\rm Rep}_k}
\newcommand{\Vecf}{{\rm Vec}_k}
\newcommand{\Qcoh}{{\rm Qcoh}}
\newcommand{\Aut}{{\rm Aut}}
\newcommand{\Dim}{{\rm dim}}
\newcommand{\End}{{\rm End}}
\newcommand{\GL}{{\rm GL}}
\newcommand{\Ob}{{\rm Ob}}
\newcommand{\Ker}{{\rm Ker}}
\newcommand{\im}{{\rm Im}}
\newcommand{\Cok}{{\rm Coker}}
\newcommand{\Det}{{\rm det}}
\newcommand{\Deg}{{\rm deg}}
\newcommand{\Sch}{{\rm Sch}_k}
\newcommand{\Var}{{\rm Var}_k}
\newcommand{\Cv}{{\rm Cv}_k}
\newcommand{\Min}{{\rm min}}
\newcommand{\uni}{{\rm uni}}
\newcommand{\calO}{\mathcal{O}}
\newcommand{\calE}{\mathcal{E}}
\newcommand{\calF}{\mathcal{F}}
\newcommand{\calC}{\mathcal{C}}

\title{AN EXTENSION OF NORI FUNDAMENTAL GROUP}
\author{Shusuke Otabe
\footnote{Mathematical Institute, Tohoku University, 6-3 Aramakiaza, Aoba, Sendai, Miyagi 980-8578, Japan;
E-mail: {sb4m06@math.tohoku.ac.jp}
}}
\affil{Mathematical Institute, Tohoku University, Sendai, Japan}
\date{\vspace{-12mm}}

\begin{document}

\maketitle

\begin{abstract}
In this paper, we study a certain extension of Nori's fundamental group in the case where a base field is of characteristic 0 and give structure theorems about it. As a result for a smooth projective curve with genus $g>1$, we prove that Nori's fundamental group acts faithfully on the category of unipotent bundles on the universal covering. In the case when $g=1$, we give a more finer result.
\end{abstract}

\begin{flushleft}

\end{flushleft}

\thispagestyle{empty}

\section{INTRODUCTION}

In ~\cite{no76}~\cite[Chapter I]{no82}, Nori defined a fundamental group scheme $\pi_1^{N}(X, x)$ for a proper reduced connected variety $X$ over a perfect field $k$ together with a rational point $x\in X(k)$. It is the Tannakian fundamental group of the category $\mathcal{C}^{N}(X)$ of \textit{essentially finite bundles} on $X$ with respect to a fiber functor $\mathcal{E}\mapsto x^{*}\mathcal{E}$.  Here an essentially finite bundle is a vector bundle trivialized by a finite torsor  $\pi:P\to X$.  Nori's fundamental group $\pi_1^N(X,x)$ is a pro-finite group scheme controlling all pointed finite torsors $(P,p)\to (X,x)$. If $k$ is an algebraically closed field of characteristic 0, then a finite torsor is nothing but a connected Galois \'etale covering over $X$, so in this case, $\pi_1^N(X,x)$ is none other than the geometric \'etale fundamental group $\pi_1(X,\overline{x})$ of $X$~\cite{gr71}. 

In \cite[Chapter II]{no82}, Nori also studied the pro-unipotent group scheme $\pi_1^{\uni}(X,x)$ controlling all \textit{unipotent torsors} over $(X,x)$. The group $\pi_1^{\uni}(X,x)$ is defined as the Tannakian fundamental group of the category $\calC^{\uni}(X)$ of \textit{unipotent bundles} on $X$ with respect to the fiber functor $x^*$. Here a vector bundle is said to be \textit{unipotent} if it is obtained by successive extensions of the trivial bundle $\mathcal{O}_X$.

 In this paper, we consider and study a bigger fundamental group $\pi_1^{EN}(X, x)$, which has, as a quotient, both $\pi_1^{N}(X,x)$ and $\pi_1^{\uni}(X,x)$, in the case where a base field is of characteristic 0. The fundamental group $\pi_1^{EN}(X, x)$ is defined as follows. A vector bundle $\mathcal{E}$ on $X$ is said to be \textit{semifinite} if it is obtained by successive extensions of essentially finite bundles. We denote by $\mathcal{C}^{EN}(X)$ the category of semifinite bundles on $X$. Then we find that  $\mathcal{C}^{EN}(X)$ is a neutral Tannakian category over $k$ (Proposition \ref{prop:cat of semi-finite}) and we define $\pi_1^{EN}(X, x)$ as its Tannakian fundamental group. If $k$ is of characteristic 0, then it turns out that the category $\mathcal{C}^{N}(X)$ is semi-simple (Remark \ref{rem:rem of semi-finite2}) and it is not closed under taking extensions in $\Coh(X)$. The category $\mathcal{C}^{EN}(X)$ is larger than $\mathcal{C}^{N}(X)$. In particular, it contains the category $\calC^{\uni}(X)$. Therefore, from the Tannaka duality, we obtain the following diagram:
\begin{equation}\label{eq:int-diag}
\begin{CD}
\pi_1^{EN}(X,x)@>p^{\uni}>>\pi_1^{\uni}(X,x)\\
@Vp^{N}VV @. @.\\
\pi_1^N(X,x)@.
\end{CD}
\end{equation}
Note that both $p^{N}$ and $p^{\uni}$ are surjective. The aim of this paper is to study $\pi_1^{EN}(X,x)$.

 Now we state the main results. First we study $\pi_1^{EN}$ in the case where $X$ is an elliptic curve (Section 3). In this case, by using Atiyah's result~\cite{at57}, we can classify completely all indecomposable semifinite bundles on $X$ (Lemma \ref{lem:semi-finite on elliptic}) and we obtain the following~(Theorem \ref{thm:main1}):

\begin{thm}\label{thm:int-main1}
Let $X$ be an elliptic curve defined over an algebraically closed field $k$ of characteristic 0 and with  $x=0\in X(k)$. Then there exists an isomorphism of group schemes
\begin{equation*}
\pi_1^{EN}(X, x)\simeq \pi_1^{N}(X, x)\times\pi_1^{\uni}(X, x).
\end{equation*}
\end{thm}

\begin{rem}\label{rem:int1}
(1) In \cite{le02}, Lekaus studied vector bundles of degree 0 on an elliptic curve~(Section 3.2) and calculated a  fundamental group of a Tannakian category generated by an indecomposable vector bundle of degree 0~(Theorem \ref{thm:lekaus}). 

(2) In \cite{la11}~\cite{la12}, Langer studied an $S$-fundamental group scheme $\pi_1^{S}(X,x)$ for a projective smooth variety in any characteristic. Here, $\pi_1^{S}(X, x)$ is the Tannakian fundamental group of the Tannakian category of \textit{Nori-semistable bundles} on $X$, which is a slightly different notion from the semistability condition considered in this paper~(Definition \ref{def:semistable}). If $k$ is of characteristic 0, $\pi_1^{EN}(X,x)$ is a quotient of $\pi_1^{S}(X,x)$. Langer gave a structure theorem of $\pi_1^{S}(X,x)$ for an abelian varirty~\cite[Theorem 6.1]{la12}. Our result for an elliptic curve (Theorem 1.1) should be compared with this theorem.
\end{rem}

Next we study $\pi_1^{EN}(X,x)$ for an arbitrary variety~(Section 4). 
In this case, it might be hard to classify all indecomposable semifinite bundles. But we can describe completely the difference between $\pi_1^{EN}$ and $\pi_1^N$. We put $\pi_1^{E}(X,x)\overset{{\rm def}}{=}\Ker(p^{N})$~(cf. (\ref{eq:int-diag})). Then we obtain the following result (Corollary \ref{cor:cor of main2}):

\begin{thm}\label{thm:int-main2}
Let $X$ be a geometrically-connected and reduced proper scheme over a field $k$ of characteristic 0. Let $\Lambda=(X_S, x_S)$ be the  pro-system of connected finite torsors over $X$ associated with $x\in X(k)$ (cf. Section 2.2). Then there is an isomorphism of group schemes
\begin{equation*}
\pi_1^{E}(X, x)\simeq\varprojlim_{\Lambda}\pi_1^{\uni}(X_S, x_S).
\end{equation*}
\end{thm}

\begin{rem}\label{rem:int-ehs}
If $k$ is a field of characteristic 0, then $\pi_1^N(X,x)$ is essentially Grothendieck's geometric fundamental group $\pi_1(X,\overline{x})$. In the case of positive characteristic, the situation is different and the former is strictly larger than the latter. This difference has been calculated by Esnault-Hai-Sun~\cite[Chapter 3]{ehs08}. Their method will be adopted to give a description of the difference between $\pi_1^{N}(X,x)$ and $\pi_1^{EN}(X,x)$~(Section 4.1).
\end{rem}

From the definition of $\pi_1^{EN}(X, x)$, we have the following exact sequence:
\begin{equation}\label{eq:int-exact seq}
1\to\pi_1^{E}(X, x)\to\pi_1^{EN}(X, x)\stackrel{p^N}{\to}\pi_1^{N}(X, x)\to 1.
\end{equation}
Assume that $k$ is an algebraically closed field of characteristic 0. Then since $\pi_1^{N}(X, x)$ is reductive and $\pi_1^{E}(X, x)$ is unipotent, the exact sequence (\ref{eq:int-exact seq}) is split and the projection $p^N$ always has a section.
Each section defines a representation of $\pi_1^{N}(X, x)$:

\begin{equation}\label{eq:int action}
\rho:\pi_1^{N}(X, x)\to\Aut(k[\pi_1^{E}(X, x)]).
\end{equation}
Then we obtain the following result (Theorem \ref{thm:main3}):

\begin{thm}\label{thm:int-main3}
Under the above notation, if $X$ is a geometrically connected proper smooth curve defined over an algebraically closed field $k$ of genus $g\ge 2$, then the representation $\rho$~(\ref{eq:int action})
is faithful.
\end{thm}

Finally, we discuss on the basic properties of $\pi_1^{EN}$~(Section 5). First recall that both Nori's fundamental group $\pi_1^{N}$ and the unipotent fundamental group $\pi_1^{\uni}$ have the base change property  for algebraic separable extensions of a base field and satisfy the  K\"unneth formula \cite{ms02}\cite{no82}~(See also Remark \ref{rem:kunneth}). Moreover, Esnault-Hai-Viehweg and Zhang studied on the homotopy sequence of Nori fundamental group $\pi_1^N$ and  they give a necessary and sufficient condition, in terms of the category of essentially finite bundles, in order that the homotopy sequence of $\pi_1^N$ for a separable proper family is exact~\cite[Theorem 3.1]{zh13}. 
In the final section, we prove the base change property and the  K\"unneth formula for our fundamental group $\pi_1^{EN}$~(Proposition \ref{prop:base change for pi1^EN} and Proposition \ref{prop:kunneth for pi1^EN}). Moreover, as the referee suggested, we include a discussion on the homotopy exact sequence for $\pi_1^{EN}$ in Section 5.2. 

Finally, we give a remark on a related work by Borne and Vistoli:

\begin{rem}\label{rem:int2}
After we submitted the present paper, the author learned that Borne-Vistoli independently carried out a similar research. We learned it a silde for a talk at a conference~\cite{bv14}.
\end{rem}

\section*{NOTATION}
Throughout this paper, $k$ always means a perfect field.  We denote by $\Sch$ the category of separated schemes of finite type over $k$. We denote by $\Var$ (resp. $\Cv$) the full subcategory of $\Sch$ consisting of geometrically-connected and reduced proper schemes over $k$ (resp.  consisting of geometrically-connected and reduced proper schemes over $k$ of dimension 1.  

Let $X\in\Ob(\Sch)$. For each $\mathcal{E}, \mathcal{F}\in\Ob(\Qcoh (X))$, we denote by $\Hom _{\calO_X}(\mathcal{E}, \mathcal{F})$ the set of morphisms in $\Qcoh (X)$ and by $\Gamma (X, \mathcal{E})$, or $\Gamma(\calE)$ the set of global sections.

We denote by $\Vecf$ the category of finite dimensional vector spaces over $k$. For an affine group scheme $G$ over $k$, we denote by $\Rep (G)$ the category of finite dimensional left linear representations of $G$ over $k$.  

Let $\mathcal{C}$ be a $k$-linear abelian category. For each family $A\subset\Ob(\mathcal{C})$, we denote by $\langle A\rangle$ the full subcategory of $\mathcal{C}$ consisting the subquotients of $P_1\oplus\dots \oplus P_s$ for some $P_1,\dots, P_s\in A$. Namely, for each $W\in\Ob(\langle A\rangle)$, there exist $V_1, V_2\in\Ob(\mathcal{C})$ and $P_1,\dots, P_s\in A$ such that $V_1\subset V_2\subset P_1\oplus\dots\oplus P_s$ and $W\simeq V_2/V_1$ in $\mathcal{C}$. Note that $\langle A\rangle$ is also a $k$-linear abelian category. If $A=\{P\}$, then we simply write $\langle P\rangle$ for $\langle\{P\}\rangle$.

Let $\mathcal{C}$ be a $k$-linear abelian rigid tensor category. For each object $P\in\Ob(\mathcal{C})$, we define the full subcategory $\langle P\rangle_{\otimes}$ of  $\mathcal{C}$ by $\langle P\rangle_{\otimes}\overset{{\rm def}}{=}\langle P^{\otimes n}\otimes (P^{\vee})^{\otimes m}|n,m>0\rangle$. Note that $\langle P\rangle_{\otimes}$ is also a $k$-linear abelian rigid tensor category.


\section{SEMIFINITE BUNDLES}

\subsection{Preliminaries}

In this subsection, we recall some basic facts of vector bundles. Let $k$ be a perfect field and $X\in\Ob(\Sch)$. A vector bundle $\mathcal{E}$ on $X$ is a locally free sheaf on $X$ of finite rank. For each vector bundle $\mathcal{E}$ on $X$, we can associate an $X$-scheme $\mathbb{V}(\mathcal{E})$ by $\mathbb{V}(\mathcal{E})\overset{{\rm def}}{=}\Spec (\mathbb{S}(\mathcal{E}^{\vee}))$, where $\mathbb{S}(\mathcal{E}^{\vee})=\bigoplus_{n\ge0}\mathbb{S}^{n}(\mathcal{E}^{\vee})$ is the \textit{symmetric algebra} of $\mathcal{E}^{\vee}$. 
This correspondence $\mathcal{E}\mapsto \mathbb{V}(\mathcal{E})$ is functorial. 
A locally free subsheaf $\mathcal{F}$ of a vector bundle $\mathcal{E}$ is called a \textit{subbundle} if the corresponding morphism $\mathbb{V}(\mathcal{F})\to\mathbb{V}(\mathcal{E})$ is a closed immersion. In this case, the quotient $\mathcal{E}/\mathcal{F}$ is also a vector bundle. 
Let $X\in\Ob(\Var)$ and let $\mathcal{V}(X)$ be the full subcategory of $\Coh (X)$ consisting of vector bundles on $X$. Now for any $\mathcal{E}, \mathcal{F}\in\Ob(\mathcal{V}(X))$ the module $\Hom_{\calO_X}(\mathcal{E}, \mathcal{F})$ is finite dimensional over $k$ and the Krull-Schmidt theorem holds in $\mathcal{V}(X)$~\cite{at56}. More precisely, any vector bundle $\mathcal{E}$ on $X$ can be written as a finite direct sum of indecomposable bundles and such a decomposition is unique up to isomorphisms.  

Let $X\in\Ob(\Cv)$ with $X$ smooth, and $\mathcal{E}$ a vector bundle on $X$ of rank $n$. The \textit{determinant} $\Det (\mathcal{E})$ of $\mathcal{E}$ is the $n$-th exterior power $\wedge^{n}{\mathcal{E}}$ of $\mathcal{E}$, which is an invertible sheaf. The \textit{degree} of $\mathcal{E}$ means the degree of $\Det (\mathcal{E})$ and the quotient $\mu(\mathcal{E})\overset{{\rm def}}{=}\Deg (\mathcal{E})/n$ is called the \textit{slope} of $\mathcal{E}$. 

\begin{definition}\label{def:semistable on curve}
Let $X\in\Ob(\Cv)$ with $X$ smooth. A  vector bundle $\mathcal{E}$ on $X$ is said to be \textit{stable} (resp. \textit{semistable}) if $\mu(\mathcal{F})< \mu(\mathcal{E})$ (resp. $\mu(\mathcal{F})\le\mu(\mathcal{E})$) for any proper nonzero subbundle $\mathcal{F}$ of $\mathcal{E}$. 
\end{definition}

Nori introduced the notion of \textit{semistable bundles} on an arbitrary proper variety:

\begin{definition}\label{def:semistable}(\cite{no76}\cite{no82})
Let $X\in\Ob(\Var)$. A \textit{curve} $C$ in $X$ is a smooth curve $C\in\Ob(\Cv)$ together with a morphism $C\to X$ which is birational onto its image. We say that a vector bundle $\mathcal{E}$ on $X$ is \textit{semistable} if for any curve $f:C\to X$ in $X$, the inverse image $f^{*}\mathcal{E}$ is semistable of slope 0 on $C$ in the sense of Definition \ref{def:semistable on curve}. We define the category $\mathcal{S}(X)$ as the full subcategory of $\Coh (X)$ consisting of semistable bundles on $X$. 
\end{definition}

Nori proved the following:

\begin{prop}(Nori)\label{semistable1}
Let $X\in\Ob(\Var)$.  Then $\mathcal{S}(X)$ is a $k$-linear abelian category. 
\end{prop}

For the proof, see ~\cite[Lemma 3.6]{no76}~\cite[Chapter I, Lemma 3.6]{no82}. We show that $\mathcal{S}(X)$ is closed under taking extensions:

\begin{prop}\label{prop:semistable2}
Let $X\in\Ob(\Var)$. Then $\mathcal{S}(X)$ is closed under taking extensions in $\Coh(X)$. More precisely, if there is an exact sequence of sheaves in $\Coh (X)$, 
\begin{equation*}
0\to\mathcal{E}'\to \mathcal{E}\to \mathcal{E}''\to 0
\end{equation*}
with $\mathcal{E}', \mathcal{E}''$ in $\mathcal{S}(X)$, then $\mathcal{E}$ also belongs to $\mathcal{S}(X)$. 
\end{prop}

\begin{proof}
Let $f:C\to X$ be a curve in $X$. Since both $\mathcal{E}'$ and $\mathcal{E}''$ are locally free, we find that $\mathcal{E}$ is also locally free. Thus the sequence
\begin{equation*}
0\to f^{*}\mathcal{E}'\to f^{*}\mathcal{E}\to f^{*}\mathcal{E}''\to 0
\end{equation*} 
is exact in $\Coh (C)$. By assumption, both $f^{*}\mathcal{E}'$ and $f^{*}\mathcal{E}''$ are semistable of slope 0 on $C$. Hence so is $f^{*}\mathcal{E}$~(\cite[Proposition 5.3.5]{le97}). This implies that $\mathcal{E}\in\Ob(\mathcal{S}(X))$.
\end{proof}


\subsection{Essentially finite bundles}

In this subsection, we recall Nori's theory on a fundamental group scheme~(cf.~\cite{no76}~\cite{no82}). Let $X\in\Ob(\Var)$. A vector bundle $\mathcal{E}$ on $X$ is said to be \textit{finite} when there exist different polynomials $f\neq g\in \mathbb{N}[t]$ suth that $f(\mathcal{E})\simeq g(\mathcal{E})$. For example, let $G$ be a finite $k$-group scheme and $\pi:P\to X$ a $G$-torsor over $X$. Then, we have $G\times P\simeq P\times_{X}P$. This implies that $(\pi_{*}\mathcal{O}_{P})^{\oplus n}\simeq (\pi_{*}\mathcal{O}_{P})^{\otimes 2}$, where $n=\Dim\Gamma(\mathcal{O}_{G})$, and $\pi_{*}\mathcal{O}_P$ is a finite bundle on X. 

Now we give two  characterizations of finite bundles. For each vector bundle $\mathcal{E}$ on $X$, we denote by $I(\mathcal{E})$ the set of isomorphism classes of direct summand of $\mathcal{E}$ and put 
\begin{equation}\label{eq:def of S}
S(\mathcal{E})\overset{{\rm def}}{=}\bigcup_{n>0}I(\mathcal{E}^{\otimes n}).
\end{equation}
Then, a vector bundle $\mathcal{E}$ on $X$ is finite if and only if $\# S(\mathcal{E})<\infty$~\cite[Lemma 3.1]{no76}~\cite[Chapter I, Lemma 3.6]{no82}. Furthermore, let $\rm{Vec}(X)$ be the set of isomorphism classes of vector bundles on $X$. We define the $\mathbb{Q}$-algebra $K(X)$ by:
\begin{equation}\label{eq:def of K}
K(X)\overset{{\rm def}}{=}(\mathbb{Z}[{\rm Vec(X)}]/H)\otimes_{\mathbb{Z}}\mathbb{Q}
\end{equation}
Here the algebra structure on $K(X)$ is induced by $[\mathcal{E}]+[\mathcal{F}]\overset{{\rm def}}{=}[\mathcal{E}\oplus\mathcal{F}], [\mathcal{E}]\cdot[\mathcal{F}]\overset{{\rm def}}{=}[\mathcal{E}\otimes\mathcal{F}]$ and $H$ is the ideal generated by all elements of the form $[\mathcal{E}\oplus\mathcal{E}']-[\mathcal{E}]-[\mathcal{E}']$. The Krull-Schmidt theorem implies that the set of isomorphism classes of indecomposable vector bundles is a basis of $K(X)$.  For each vector bundle $\mathcal{E}$, we denote by $R(\mathcal{E})$ the subalgebra of $K(X)$ generated by $S(\mathcal{E})$. Then $\mathcal{E}$ is finite if and only if the Krull dimension of $R(\mathcal{E})$ is equal to  0~\cite[Lemma 3.1]{no76}~\cite[Chapter I, Lemma 3.6]{no82}.  

The first characterization implies that a line bundle is finite if and only if it is a torsion line bundle~\cite[Lemma 3.2]{no76}~\cite[Chapter I, Lemma 3.2]{no82}.
The second one implies that a direct sum and a tensor product of any two finite bundles are also finite, and that the dual and a direct summand of a finite bundle are also finite.
 Nori showed that any finite bundle is semistable~\cite[Corollary 3.5]{no76}~\cite[Chapter I, Corollary 3.5]{no82}. Therefore, the following is well-defined:

\begin{definition}
Let $A\subset\Ob(\mathcal{S}(X))$ be the family of finite indecomposable bundles on $X$. We define the category $\mathcal{C}^{N}(X)$ by $\mathcal{C}^{N}(X)\overset{{\rm def}}{=}\langle A\rangle\subset\mathcal{S}(X)$. An semistable bundle $\calE$ on $X$ is said to be \textit{essentially finite bundle} on $X$ if $\calE$ belongs to $\calC^N(X)$. 
\end{definition}

Nori proved that if $X(k)\neq \emptyset$, then $\mathcal{C}^{N}(X)$ is a neutral Tannakian category over $k$ in the sense of \cite{dm82}. In this case, each rational point $x\in X(k)$ defines a neutral fiber functor $\omega_{x}\overset{{\rm def}}{=}x^{*}:\mathcal{C}^{N}(X)\to \Vecf; \mathcal{E}\mapsto x^{*}\mathcal{E}$. Note that the condition that the connectivity of  $X$ makes $\omega_x$ a faithful functor. Nori defined the fundamental group scheme of $X$ with respect to a base point $x$ as the Tannakian fundamental group of $\mathcal{C}^{N}(X)$ with a fiber functor $\omega_{x}$. We denote it by $\pi_1^{N}(X,x)$. 
For each finite bundle $\mathcal{E}$, since $S(\mathcal{E})$ is a finite set, the corresponding Tannakian fundamental group $G$ of  $\langle\mathcal{E}\rangle_{\otimes}$ is a finite $k$-group scheme ~\cite[Proposition 2.20(1)]{dm82}. Thus $\pi_1^N(X,x)$ is a profinite group scheme over $k$. 

Let $X\in\Ob(\Sch)$ and $G$ an affine $k$-group scheme. A $G$-\textit{torsor} over $X$ is an $X$-scheme $\pi:P\to X$ together with a morphism $\phi:P\times G\to P$ satisfying the following properties:

(a) $\pi$ is a faithfully flat affine morphism,

(b) $\phi:P\times G\to P$ defines an action of $G$ on $P$ with $\pi\circ\phi=\pi\circ {\rm pr}_1$,

(c) $({\rm pr}_1, \phi):P\times G\to P\times_X P$ is an isomorphism.

To each $G$-torsor $\pi:P\to X$, we can associate an exact tensor functor $F(P):\Rep (G)\to \Qcoh (X)$ as follows. Each representation $V$ of $G$ gives rise to a vector bundle $(\mathcal{O}_P\otimes_k V)/G$ on $X$. This correspondence $V\mapsto (\mathcal{O}_P\otimes_k V)/G$ defines an exact　faithful tensor functor $\Rep (G)\to \Qcoh (X)$. We denote it by $F(P)$. Nori proved that every exact faithful tensor functor $\Rep (G)\to \Qcoh (X)$ is of the form $F(P)$ for some $G$-torsor $P\to X$. Furthermore, for each rational point $x\in X(k)$, the composition $x^{*}\circ F(P):\Rep(G)\to\Vecf$ is nothing but the forgetful functor and the corresponding torsor $P\times_X x\to x$ is trivial. This implies that the torsor $P$ admits a rational point $p$ above $x$. For details, see  \cite[Proposition 2.9]{no76}~\cite[Chapter 1, Proposition 2.9]{no82}. 

Now we apply this to our setting. Let $X\in\Ob(\Var)$ with $X(k)\neq \emptyset$. Fix a rational point $x\in X(k)$. By Tannaka duality, the fiber functor $\omega_x:\mathcal{C}^{N}(X)\to \Vecf$ induces an equivalence of categories
\begin{equation}\label{eq:equiv}
\mathcal{C}^{N}(X)\stackrel{\simeq}{\to}\Rep (\pi_1^{N}(X, x)).
\end{equation}
Let $F$ be the inverse functor of (\ref{eq:equiv}).  From the above discussion, there exists a pointed $\pi_1^{N}(X, x)$-torsor $\pi:(X^N_x, x^N)\to (X,x)$ with $F=F(X_x^N)$.  Let $S$ be a full tensor subcategory of $\mathcal{C}^{N}(X)$ generated by finitely many objects and denote by $\pi(X, S, x)$ its Tannakian fundamental group, which is a finite $k$-group scheme and 
$\pi_1^{N}(X, x)=\varprojlim_{S}\pi(X, S, x)$. 
By composing with $F$, the natural inclusion $\Rep (\pi(X, S, x))\hookrightarrow\Rep (\pi_1^{N}(X, x))$ yields a $\pi(X, S, x)$-torsor  $\pi_S:(X_S,x_S)\to (X,x)$ together with a rational point $x_S\in X_S(k)$ above $x$. Since $F(X_S)$ is fully faithful, we have
\begin{equation}\label{eq:geom-conn}
\begin{aligned}
\Gamma(X_S, \mathcal{O}_{X_S})&=\Hom_{\calO_{X_S}}(\mathcal{O}_{X_S}, \mathcal{O}_{X_S})=\Hom_{\calO_X}(\mathcal{O}_X, (\pi_S)_{*}\mathcal{O}_{X_S})\\
&\simeq\Hom_{\Rep (\pi(X, S, x))}(k, k[\pi(X, S, x)])=(k[\pi(X, S, x)])^{\pi(X, S, x)}=k,
\end{aligned}
\end{equation}
so $X_S$ is geometrically-connected~\cite[Chapter II, Proposition 3]{no82}. Furthermore, for any object $\mathcal{E}$ in $\mathcal{C}^{N}(X)$, we have
\begin{equation}\label{eq:section 2 equiv}
\mathcal{E}\in \Ob(S)\Leftrightarrow \pi_S^{*}\mathcal{E}=\text{trivial}.
\end{equation}
Let $S\subseteq T\subset\mathcal{C}^{N}(X)$ be two finitely generated full tensor subcategories. Then there is a unique morphism $\pi_{TS}:X_T\to X_S$ over $X$ such that $\pi_{TS}(x_T)=x_S$. Note that $\pi_{TS}:X_T\to X_S$ is a finite torsor with strucure group $G_{TS}\overset{{\rm def}}{=}\Ker(\pi(X, T, x))\twoheadrightarrow\pi(X, S, x))$~\cite[Lemma 2.6]{ehs08}. We have a pro-system $\Lambda=(X_S, x_S)_S$ of connected finite torsors over $X$ with $(X_x^N,x^N)=\varprojlim_S(X_S,x_S)$. Note that  If $\pi(X,S,x)$ is \'etale, then $X_S$ is geometrically-reduced.

\begin{rem}\label{rem:universality of pi_1^N}
The fundamental group $\pi_1^N(X,x)$ has the following universal property: for any finite group scheme $G$ over $k$, the natural map
\begin{equation*}
\Hom(\pi_1^N(X,x),G)\to{{\rm Tors}((X,x),G)};\,\phi\mapsto (P,p)=(X^N_x\times G)/\pi_1^N(X,x),
\end{equation*}
is bijective. Here ${{\rm Tors}}((X,x),G)$ stands for the set of isomorphism classes of all pointed $G$-torsors $(P,p)$ over $(X,x)$.
\end{rem}

\begin{rem}\label{rem:semi-simple}
If the characteristic of $k$ is 0, then every finite $k$-group scheme $G$ is reduced and $\Rep(G)$ is semi-simple. Thus, in this case, $\mathcal{C}^{N}(X)$ is semi-simple and essentially finite bundles are indeed finite~\cite[Section 3]{no76}~\cite[Chapter I, 3]{no82}.

Furthermore, if $k$ is an algebraically closed field of characteristic 0, then a finite $G$-torsor $(P,p)\to (X,x)$ together with a point above $x$ is nothing but a connected Galois \'etale covering of $(X,x)$ with Galois group $G(k)$ and the fundamental group scheme $\pi_1^N(X,x)$ is none other than the geometric \'etale fundamental group $\pi_1(X,\overline{x})$ of $X$.
\end{rem}


\subsection{Unipotent bundles}

In \cite[Chapter IV]{no82}, Nori also considered the category of \textit{unipotent bundles}.
Let $X\in\Ob(\Var)$.

\begin{definition} A vector bundle $\mathcal{E}$ on $X$ is said to be \textit{unipotent} if there exists a filtration 
\begin{equation}
\mathcal{E}=\mathcal{E}^{(0)}\supset\mathcal{E}^{(1)}
\supset\dots\supset\mathcal{E}^{(n)}=0
\end{equation}
such that $\mathcal{E}^{(i)}/\mathcal{E}^{(i+1)}\simeq \mathcal{O}_X$ for any $i$. 
We define the category $\mathcal{C}^{\uni}(X)$ as the full subcategory of $\Coh (X)$ consisting of unipotent bundles on $X$. 
\end{definition}

By Propositon \ref{prop:semistable2}, the category $\mathcal{C}^{\uni}(X)$ is contained in $\mathcal{S}(X)$. Nori proved that $\mathcal{C}^{\uni}(X)$ is a Tannakian category~\cite[Chapter IV, Lemma 2]{no82}:

\begin{prop}\label{prop:nilpotent}
(Nori) Let $X\in\Ob(\Var)$ with $X(k)\neq\emptyset$. Then the category $\mathcal{C}^{\uni}(X)$ is a neutral Tannakian category over $k$ and its Tannakian fundamental group is a \textit{unipotent} affine group scheme over $k$.
\end{prop}

Here, an affine group scheme $G$ is said to be unipotent if any non-trivial representation $\rho:G\to \GL (V)$ has an element $0\neq v\in V$ such that $gv=v$ for any $g\in G$. 
For the proof of Proposition 2.9, see \cite[Chapter IV.1]{no82}. Again, each rational point $x\in X(k)$ defines a neutral fiber functor $\omega_x:\mathcal{C}^{\uni}(X)\to\Vecf;\mathcal{E}\mapsto x^{*}\mathcal{E}$ of $\mathcal{C}^{\uni}(X)$. We denote by $\pi_1^{\uni}(X, x)$ the Tannakian fundamental group of $\mathcal{C}^{\uni}(X)$ with respect to the fiber functor $\omega_x$. 

\begin{prop}\label{prop:n-nil}
Let $X\in\Ob(\Var)$. If $k$ is of characteristic 0, then the following hold:

(1) Any object of $\mathcal{C}^{N}(X)\cap \mathcal{C}^{\uni}(X)$ is isomorphic to $\mathcal{O}_X^{\oplus d}$ for some $d\ge 0$. 

(2) For any finite bundles $\mathcal{F}$, $\mathcal{F}'\in \Ob(\mathcal{C}^{N}(X))$ and any unipotent bundles $\mathcal{E}$, $\mathcal{E}'\in \Ob(\mathcal{C}^{\uni}(X))$, the natural $k$-linear map
\begin{equation*}
\Hom_{\mathcal{O}_X}(\mathcal{F}, \mathcal{F}')\otimes \Hom_{\mathcal{O}_X}(\mathcal{E}, \mathcal{E}')\to \Hom_{\mathcal{O}_X}(\mathcal{F}\otimes \mathcal{E}, \mathcal{F}'\otimes \mathcal{E}')
\end{equation*}
is bijective.
\end{prop}

\begin{proof}
(1) Note that both $\mathcal{C}^{N}(X)$ and $\mathcal{C}^{\uni}(X)$ are closed under taking direct summands in $\mathcal{S}(X)$ and so is $\mathcal{C}^{N}(X)\cap\mathcal{C}^{\uni}(X)$. Thus it suffices to deal with only indecomposable bundles. Let $\mathcal{E}$ be an indecomposable bundle in $\mathcal{S}(X)$ which belongs to both $\mathcal{C}^{N}(X)$ and $\mathcal{C}^{\uni}(X)$. Since $\mathcal{C}^{N}(X)$ is semi-simple (Remark \ref{rem:semi-simple}), we find that $\mathcal{E}$ is a simple object in $\mathcal{C}^{N}(X)$. On the other hand, since $\mathcal{E}$ is unipotent and there is an inclusion $\mathcal{O}_X\hookrightarrow \mathcal{E}$. Since both $\mathcal{O}_X$ and $\mathcal{E}$ are simple in $\mathcal{C}^{N}(X)$, this inclusion must be an isomorphism of sheaves. Thus any indecomposable bundle which belongs to $\mathcal{C}^{N}(X)\cap\mathcal{C}^{\uni}(X)$ must be isomorphic to $\mathcal{O}_X$.   

(2) For any two vector bundles $\mathcal{E}$, $\mathcal{E}'$, we have $\Hom_{\calO_X}(\mathcal{E}, \mathcal{E}')\simeq \Hom_{\calO_X}(\mathcal{O}_X, \mathcal{E}^{\vee}\otimes \mathcal{E}')=\Gamma(\mathcal{E}^{\vee}\otimes \mathcal{E}')$. Thus it suffices to show that for any $\mathcal{F}\in\Ob(\mathcal{C}^{N}(X))$ and $\mathcal{E}\in\Ob(\mathcal{C}^{\uni}(X))$, the natural map
\begin{equation*}
\Gamma(\mathcal{F})\otimes \Gamma(\mathcal{E})\to \Gamma(\mathcal{F}\otimes\mathcal{E})
\end{equation*}
is bijective. We may assume that $\mathcal{F}$ and $\mathcal{E}$ are indecomposable in $\mathcal{C}^{N}(X)$ and $\mathcal{C}^{\uni}(X)$, respectively. If $\mathcal{F}=\mathcal{O}_X$, then the above map is reduced to the natural bijection $\Gamma(\mathcal{F})\otimes_k k\simeq\Gamma(\mathcal{F})$ and the statement is true. It remains to be proven that if $\mathcal{F}$ is not isomorphic to the trivial bundle $\mathcal{O}_X$, then $\mathcal{F}\otimes\mathcal{E}$ has no global sections. Take a filtration 
\begin{equation*}
\mathcal{E}=\mathcal{E}^{(0)}\supset \mathcal{E}^{(1)}\supset\dots\supset\mathcal{E}^{(n)}=0
\end{equation*}
such that $\mathcal{E}^{(i)}/\mathcal{E}^{(i+1)}\simeq \mathcal{O}_X$ for any $i$. Then the filtration obtained by tensoring with $\mathcal{F}$, 
\begin{equation*}
\mathcal{F}\otimes\mathcal{E}=\mathcal{F}\otimes\mathcal{E}^{(0)}\supset \mathcal{F}\otimes\mathcal{E}^{(1)}\supset\dots\supset\mathcal{F}\otimes\mathcal{E}^{(n)}=0
\end{equation*}
gives a Jordan-H\"older filtration of $\calF\otimes\calE$ in $\mathcal{S}(X)$ and all the subquotients  $(\mathcal{F}\otimes\mathcal{E}^{(i)})/(\mathcal{F}\otimes\mathcal{E}^{(i+1)})\simeq \mathcal{F}$ are not isomorphic to the trivial bundle $\mathcal{O}_X$ for any $i$. Thus $\mathcal{O}_X$ cannot be embedded into $\mathcal{F}\otimes\mathcal{E}$, whence $\Gamma(\mathcal{F}\otimes\mathcal{E})=0$. This finishes the proof.
\end{proof}

\begin{rem}
In the case of positive characteristic, the situation is quite different. In fact, $\mathcal{C}^{\uni}(X)$ is a full subcategory of $\mathcal{C}^{N}(X)$~\cite[Chapter IV.1, Proposition 3]{no82}.
\end{rem}


\subsection{Semifinite bundles}

In this subsection, we introduce the notion of \textit{semifinite bundles} in the case where a base field is of characteristic 0.
Let $k$ be a field of characteristic 0 and $X\in\Ob(\Var)$ with $X(k)\neq\emptyset$.

\begin{definition}
A vector bundle $\mathcal{E}$ on $X$ is said to be \textit{semifinite} if there exists a filtration
\begin{equation*}
\mathcal{E}=\mathcal{E}^{(0)}\supset \mathcal{E}^{(1)}\supset\dots\supset\mathcal{E}^{(n)}=0
\end{equation*}
such that $\mathcal{E}^{(i)}/\mathcal{E}^{(i+1)}$ is finite indecomposable for any $i$. We denote by $\mathcal{C}^{EN}(X)$ the full subcategory of $\Coh (X)$ consisting of semifinite bundles on $X$. 
\end{definition}

From Proposition \ref{prop:semistable2}, the category $\mathcal{C}^{EN}(X)$ is a subcategory of $\mathcal{S}(X)$. 
By definition, the category $\mathcal{C}^{EN}(X)$ is closed under taking extensions in $\Coh (X)$.

\begin{rem} Under the above notation,

(1) we have $\mathcal{C}^{N}(X), \mathcal{C}^{\uni}(X)\subseteq\mathcal{C}^{EN}(X)\subset\mathcal{S}(X)$, and 

(2) any simple object in $\mathcal{C}^{EN}(X)$ is nothing but a finite indecomposable bundle on $X$.
\end{rem}

Now we obtain the following:

\begin{prop}\label{prop:cat of semi-finite}
The category $\mathcal{C}^{EN}(X)$ is a $k$-linear abelian rigid tensor category. Furthermore, each rational point $x\in X(k)$ defines a neutral fiber functor $\omega_x:\mathcal{C}^{EN}(X)\to\Vecf;\mathcal{E}\mapsto x^{*}\mathcal{E}$ of $\mathcal{C}^{EN}(X)$. Therefore, under the condition that $X(k)\neq\emptyset$, it is a neutral Tannakian category over $k$.
\end{prop}

\begin{proof}
We adopt the proof of \cite[Chapter IV.1]{no82}.
First we show the category $\mathcal{C}^{EN}(X)$ is a rigid tensor category. Let $\calE,\,\calF\in\Ob(\mathcal{C}^{EN}(X))$. We must show that $\calE\otimes\calF,\,\calE^{\vee}\in\Ob(\calC^{EN}(X))$. If $\calE$ is finite, then by induction on $\rank\calF$, we can find that $\calE\otimes\calF\in\Ob(\calC^{EN}(X))$. If $\calE$ is an arbitrary semifinite bundle, take a filtration
\begin{equation*}
\calE=\calE^{(0)}\supset\calE^{(1)}\supset\cdots\supset\calE^{(n)}=0
\end{equation*}
with $\calE^{(i)}/\calE^{(i+1)}$ finite bundles. By tensoring with $\calF$, we obtain the filtration on $\calE\otimes\calF\supseteq\calE^{(i)}\otimes\calF\supset
\calE^{(i+1)}\otimes\calF$ with $\calE^{(i)}\otimes\calF/\calE^{(i+1)}\otimes\calF\simeq
(\calE^{(i)}/\calE^{(i+1)})\otimes\calF$ semifinite. Since $\calC^{EN}(X)$ is closed under taking extensions, we can conclude that $\calE\otimes\calF\in\Ob(\calC^{EN}(X))$. On the other hand, we define the subbundle $(\calE^{\vee})^{(i)}\subset\calE^{\vee}$ by $(\calE^{\vee})^{(i)}\overset{{\rm def}}{=}(\calE/\calE^{(n-i)})^{\vee}$. Then these gives a filtration on $\calE^{\vee}$ with $(\calE^{\vee})^{(i)}/(\calE^{\vee})^{(i+1)}\simeq(\calE^{(n-i-1)}/\calE^{(n-i)})^{\vee}$ finite. Therefore, $\calE^{\vee}$ is semifinite.

Finally, we will show that it is an abelian category. 
Let $f:\mathcal{E}\to\mathcal{F}$ be a morphism with $\mathcal{E}$ and $\mathcal{F}$ in $\mathcal{C}^{EN}(X)$. We must show that both $\Ker(f)$ and $\Cok(f)$ are in $\mathcal{C}^{EN}(X)$. We will do that by induction on $\rank(\mathcal{E})+\rank(\mathcal{F})$. If $\rank(\mathcal{E})=0$, or $\rank(\mathcal{F})=0$, then it is clear that $\Ker(f), \Cok(f)\in \Ob(\mathcal{C}^{EN}(X))$. If both $\mathcal{E}$ and $\mathcal{F}$ are simple in $\mathcal{C}^{EN}(X)$, then in particular these are finite bundles and $\Ker(f), \Cok(f)\in \Ob(\mathcal{C}^{N}(X))\subset \Ob(\mathcal{C}^{EN}(X))$. (In fact, in this case, $f$ is zero, or an isomorphism.) 
Assume that $\mathcal{E}$, or $\mathcal{F}$ is not simple in $\mathcal{C}^{EN}(X)$. By considering the dual $f^{\vee}:\mathcal{F}^{\vee}\to \mathcal{E}^{\vee}$, if necessary, we may assume that $\mathcal{F}$ is not simple. Then there exists an exact sequence
\begin{equation*}
0\to \mathcal{F}' \stackrel{i}{\to}\mathcal{F}\stackrel{j}{\to}\mathcal{F}''\to 0,
\end{equation*}
where $\mathcal{F}',\mathcal{F}''\in\Ob(\mathcal{C}^{EN}(X))$ with $\mathcal{F}''$ simple. Note that $\rank(\mathcal{E})+\rank(\mathcal{F}')<\rank(\mathcal{E})+\rank(\mathcal{F})$. Since $\mathcal{F}''$ is simple, $\im (j\circ f)$ is 0, or surjective. In the first case, $f$ factors through $\mathcal{F}'$,i.e., there exists a morphism $g:\mathcal{E}\to \mathcal{F}'$ such that $i\circ g=f$. 
\begin{equation*}
\begin{CD}
0@>>>\mathcal{E}@=\mathcal{E}@>>> 0 @.\\
@. @VgVV @VfVV @VVV @.\\
0@>>>\mathcal{F}'@>i>>\mathcal{F}@>j>>\mathcal{F}''@>>> 0
\end{CD}
\end{equation*}
By the induction hypothesis, $\Ker(g)$ and $\Cok(g)$ are in $\mathcal{C}^{EN}(X)$. By the snake lemma, we have $\Ker(f)\simeq\Ker(g)$ and the exact sequence $0\to\Cok(g)\to\Cok(f)\to\mathcal{F}''\to0$, which implies that $\Ker(f)$ and $\Cok(f)$ are in $\mathcal{C}^{EN}(X)$.
In the second case,i.e., $\im(j\circ f)=\mathcal{F}''$, again the snake lemma implies that both $\Ker(f)$ and $\Cok(f)$ are in $\mathcal{C}^{EN}(X)$, which completes the proof.
\end{proof}

\begin{rem}\label{rem:rem of semi-finite1}
Let $S$ be a full tensor abelian subcategory of $\mathcal{C}^{N}(X)$. We define the category $\overline{S}$ as the full subcategory of $\mathcal{C}^{EN}(X)$ consisting of semifinite bundles obtained by taking successive extensions of finite bundles in $S$. The same proof as the above proposition implies that $\overline{S}$ is a $k$-linear abelian rigid tensor category. Furthermore, it is closed under taking subquotients in $\mathcal{C}^{EN}(X)$. 
\end{rem}

\begin{definition} Under the above notation, 
we denote by $\pi_1^{EN}(X, x)$ the Tannakian fundamental group of $\mathcal{C}^{EN}(X)$ with respect to the fiber functor $\omega_x$. 
\end{definition}

\begin{rem}\label{rem:rem of semi-finite2}
Note that both $\mathcal{C}^{N}(X)$ and $\mathcal{C}^{\uni}(X)$ are full subcategories of $\mathcal{C}^{EN}(X)$ which are closed under taking subquotients, and, in particular, are closed under taking subobjects. Thus, by \cite[Proposition2.12]{dm82} or Lemma \ref{lem:exactness} (i), we obtain the diagram (\ref{eq:int-diag}).
\end{rem}


\section{SEMIFINITE BUNDLES ON AN ELLIPTIC CURVE}

In this section, we assume that $k$ is an algebraically closed field of characteristic 0 and  that $X$ is an elliptic curve over $k$, i.e., a smooth curve $X$ in $\Cv$ of genus 1 with $x\in X(k)$ fixed.


\subsection{Vector bundles of degree 0 on an elliptic curve}

In this subsection, we recall Atiyah's theorem for vector bundles on an elliptic curve\cite{at57} and Lekaus' work for vector bundles of degree 0 on an elliptic curve~\cite{le02}.
In \cite{at57}, Atiyah studied and classified  vector bundles on an elliptic curve. The following is a part of his results:

\begin{thm}\label{thm:atiyah}
(Atiyah) Let $E(r,0)$ be the set of isomorphism classes of indecomposable vector bundles of rank $r$ and of degree 0 on $X$.

(1) There exists a vector bundle $\mathcal{E}_r\in E(r,0)$, unique up to isomorphisms, with $\Gamma(\mathcal{E}_r)\neq 0$. Moreover we have an exact sequence:
\begin{equation*}
0\to\mathcal{O}_X\to\mathcal{E}_r\to\mathcal{E}_{r-1}\to 0.
\end{equation*}

(2) Let $\mathcal{E}\in E(r,0)$, then $\mathcal{E}\simeq \mathcal{E}_r\otimes \mathcal{L}$ for some line bundle $\mathcal{L}$  of degree 0 with $\mathcal{L}\simeq \Det\mathcal{E}$. 

(3) Each $\mathcal{E}_r$ is self dual,i.e., $\mathcal{E}_r^{\vee}\simeq \mathcal{E}_r$.

(4) We have $\mathcal{E}_r\otimes \mathcal{E}_s\simeq \bigoplus_{i=1}^{\Min(r,s)}\mathcal{E}_{r_i}$. Furthermore, we obtain
\begin{equation*}
\mathcal{E}_r\otimes \mathcal{E}_s\simeq \mathcal{E}_{r-s+1}\oplus\mathcal{E}_{r-s+3}\oplus\dots\oplus\mathcal{E}_{(r-s)+(2s-1)}
\end{equation*}
for $2\le s\le r$.

(5) We have $\Dim\Gamma(\mathcal{E}_r\otimes\mathcal{E}_s)=\Min(r, s)$. In particular, we have $\Dim\Gamma(\mathcal{E}_r)=1$. 
\end{thm}

In \cite{le02}, Lekaus studied Tannakian categories generated by indecomposable vector bundles of degree 0 on $X$. For each indecomposable vector bundle $\mathcal{E}$ of degree 0, we define a full subcategory $\mathcal{C}(\mathcal{E})$ of $\mathcal{S}(X)$ by $\mathcal{C}(\mathcal{E})\overset{{\rm def}}{=}\langle S(\mathcal{E})\rangle$(cf. (\ref{eq:def of S})). 
Lekaus proved that for any $\mathcal{E}\in E(r,0)$, the category $\mathcal{C}(\mathcal{E})$ is  a neutral Tannakian category over $k$ and  calculated these Tannakian fundamental groups by using Theorem \ref{thm:atiyah} and obtained the following results:

\begin{thm}\label{thm:lekaus}
(Lekaus)
Under the above notation, the following hold:

(1) The Tannakian fundamental group of $\mathcal{C}(\mathcal{E}_2)$ is isomorphic to $\mathbb{G}_a$. 

(2) If $\mathcal{L}$ is a non-torsion line bundle, then the Tannakian fundamental group of $\mathcal{C}(\mathcal{E}_2\otimes \mathcal{L})$ is isomorphic to $\mathbb{G}_a\times \mathbb{G}_m$.

(3) If $\mathcal{L}$ is a torsion line bundle of order $m>0$, then the Tannakian fundamental group  of $\mathcal{C}(\mathcal{E}_2\otimes\mathcal{L})$ is isomorphic to $\mathbb{G}_a\times \mu_m$. 
\end{thm}

\begin{rem}\label{rem:rem of lekaus's work1}
(1) By using Atiyah's theorem (Theorem \ref{thm:atiyah}(1), (3) and (4)), Lekaus showed that 
\begin{equation*}
S(\mathcal{E}_r)=\begin{cases}
                                  \{\mathcal{E}_k;k>0\}&\text{if $r$: even,}\\
                                  \{\mathcal{E}_{2k-1};k>0\}&\text{if $r$: odd.}
                                 \end{cases}
\end{equation*}
However, for any $k>0$, $\mathcal{E}_{2k-1}$ is a subbundle of $\mathcal{E}_{2k}$ (Theorem 3.1(1), (3)). Thus we have $\mathcal{C}(\mathcal{E}_r)=\mathcal{C}(\mathcal{E}_2)$ for any $r>1$.

(2) Let $\mathcal{L}$ be a torsion line bundle of order $m>0$. Lekaus also proved that
\begin{equation*}
S(\mathcal{E}_2\otimes\mathcal{L})=\begin{cases}
                                  \{\mathcal{E}_{2k-1}\otimes\mathcal{L}^{\otimes 2i}, \mathcal{E}_{2k}\otimes\mathcal{L}^{\otimes 2i+1};k>0, i=0,\dots,m/2\}&\text{if $m$: even,}\\
                                  \{\mathcal{E}_{2k-1}\otimes\mathcal{L}^{\otimes i};k>0, i=0,\dots, m-1\}&\text{if $m$: odd.}
                                 \end{cases}
\end{equation*}
\end{rem}

\begin{rem}\label{rem:rem of lekaus's work2}
If $\mathcal{E}$ is a finite bundle, then Nori's theory shows that there is a finite $G_{\mathcal{E}}$-torsor $\pi:P_{\calE}\to X$ such that $\pi^{*}(\mathcal{E})$ is trivial over $P$. In this case, the equality $\Dim R(\mathcal{E})=\Dim G_{\mathcal{E}}(=0)$ holds~(cf.~Section 2.2). In \cite{le02}, Lekaus showed that such relation between $\Dim R(\mathcal{E})$ and $\Dim G_{\mathcal{E}}$ can be generalized for  a vector bundle of degree 0 on an elliptic curve. Namely, Lekaus proved that for each $\mathcal{E}\in E(r, 0)$, there exists a  $G_{\mathcal{E}}$-torsor $\pi:P_{\calE}\to X$ such that $\pi^{*}\mathcal{E}=\text{trivial}$ and $\Dim R(\mathcal{E})=\Dim G_{\mathcal{E}}(=1)$ and $\mathcal{E}$ can not be trivialized by any  $G$-torsor with  $\Dim G<\Dim R(\mathcal{E})$. Thus, Theorem \ref{thm:lekaus} implies that if $\rank\mathcal{E}\ge 2$, $\Dim R(\mathcal{E})$ can not be 0. 
\end{rem}

\begin{rem}\label{rem:finite bundle on elliptic curve}
Remark \ref{rem:rem of lekaus's work2} shows that finite indecomposable bundles on $X$ must be line bundles~(cf.~Section 2.2). 
However, a line bundle is finite if and only if it is a torsion line bundle. Therefore, a finite indecomposable bundle on an elliptic curve $X$ is nothing but a torsion line bundle on $X$. 
\end{rem}


\subsection{Structure theorem of $\pi_1^{EN}$ for an elliptic curve}

In this subsection, we give a proof of Theorem \ref{thm:int-main1}, which is a generalization of Theorem \ref{thm:lekaus} (3).

\begin{lem}\label{lem:nil on elliptic}
(1) All indecomposable bundles of $\mathcal{C}(\mathcal{E}_2)$ are of the form $\mathcal{E}_r$.

(2) We have $\mathcal{C}(\mathcal{E}_2)=\mathcal{C}^{\uni}(X)$. 
\end{lem}

\begin{proof}
(1) Note that Theorem \ref{thm:atiyah}(1) implies that all $\mathcal{E}_r (r>0)$ are  unipotent bundles.
Take an indecomposable bundle $\mathcal{E}$ which is in $\mathcal{C}(\mathcal{E}_2)$. Then it is a subquotient of a finite direct sum of $\mathcal{E}_{r_i}, i=1,\dots,t$ (Remark \ref{rem:rem of lekaus's work1}). Thus there exists an indecomposable bundle $\mathcal{F}$ of degree 0 such that $\Hom_{\calO_X}(\mathcal{F},\mathcal{E})\neq 0$ and $\Hom_{\calO_X}(\mathcal{F}, \oplus_{i=1}^{t}\mathcal{E}_{r_i})\neq 0$.  If $\rank\mathcal{F}=r$, then we have $\mathcal{F}=\mathcal{E}_r\otimes\mathcal{L}$ for some line bundle $\mathcal{L}$ (Theorem \ref{thm:atiyah}(2)). The condition $\Hom_{\calO_X}(\mathcal{F}, \oplus_{i=1}^{t}\mathcal{E}_{r_i})\neq 0$ implies that 
\begin{equation*}
0\neq\Hom_{\calO_X}(\mathcal{F}, \mathcal{E}_{r_i})\simeq\Hom_{\calO_X}(\mathcal{E}_r, \mathcal{E}_{r_i})\otimes\Hom_{\calO_X}(\mathcal{L}, \mathcal{O}_X)
\end{equation*}
for some $i$ and $\mathcal{F}$ must be of the form $\mathcal{E}_r$ (cf. Proposition \ref{prop:n-nil}(2)). Thus we have $\Hom_{\calO_X}(\mathcal{E}_r,\mathcal{E})\neq 0$, which again implies that $\mathcal{E}$ must be of the form $\mathcal{E}_s$. 

(2) Let $\mathcal{E}\in\Ob(\mathcal{C}^{\uni}(X))$. We must show that $\mathcal{E}\in\Ob(\mathcal{C}(\mathcal{E}_2))$. Since $\mathcal{C}^{\uni}(X)$ is closed under taking direct summands, we may assume that $\mathcal{E}$ is indecomposable. Since $\mathcal{E}$ is unipotent, there exists an inclusion $\mathcal{O}_X\hookrightarrow\mathcal{E}$, which implies that $\Gamma(\mathcal{E})\neq 0$.
Therefore, by Theorem \ref{thm:atiyah}(1), $\mathcal{E}$ must be of the form $\mathcal{E}_r$.  Thus it is  in $\mathcal{C}(\mathcal{E}_2)$ (Remark \ref{rem:rem of lekaus's work1}(1)). 
\end{proof}

\begin{lem}\label{lem:semi-finite on elliptic}
(1) Let $\mathcal{L}$ be a torsion line bundle of order $m>0$. Then any indecomposable bundle in $\mathcal{C}(\mathcal{E}_2\otimes\mathcal{L})$ is of the form $\mathcal{E}_r\otimes\mathcal{L}^{\otimes i}$ for some $r>0$ and $i=0,\dots,m-1$. 

(2) For any torsion line bundle $\mathcal{L}$, the category $\mathcal{C}(\mathcal{E}_2\otimes\mathcal{L})$ is a full subcategory of $\mathcal{C}^{EN}(X)$ which is closed under taking subobjects.

(3) Any indecomposable bundles in $\mathcal{C}^{EN}(X)$ is of the form $\mathcal{E}_r\otimes\mathcal{L}$ for some $r>0$ and for some  torsion line bundle $\mathcal{L}$. 
\end{lem}

\begin{proof}
(1) Let $\mathcal{E}$ be an indecomposable bundle in $\mathcal{C}(\mathcal{E}_2\otimes\mathcal{L})$. By Theorem \ref{thm:atiyah}(2), we can write $\mathcal{E}=\mathcal{E}_r\otimes\mathcal{L}'$ for some $r>0$ and some line bundle $\mathcal{L}'$ of degree 0. Since $\mathcal{E}$ is a subquotient of $\oplus_{j=1}^{t}P_j$ for some $P_j\in S(\mathcal{E}_2\otimes\mathcal{L})$, there exists an indecomposable bundle $\mathcal{F}$ in $\mathcal{S}(X)$ such that $\Hom_{\calO_X}(\mathcal{F},\mathcal{E})\neq 0$ and $\Hom_{\calO_X}(\mathcal{F}, \oplus_{j=1}^{t}P_j)\neq 0$. From the same reason as in the proof of Lemma  \ref{lem:nil on elliptic}(1), we can find that $\mathcal{L}'$ must be of the form $\mathcal{L}^{\otimes i}$ for some $i=0,\dots,m-1$(Remark \ref{rem:rem of lekaus's work1}(2)). This finishes the proof.

(2) Assume that the order of $\mathcal{L}$ is $m>0$. For any $i=1,\dots,m$ and $r>0$, the bundle $\mathcal{E}\overset{{\rm def}}{=}\mathcal{E}_r\otimes\mathcal{L}$ has a filtration 
\begin{equation*}
\mathcal{E}=\mathcal{E}^{(0)}\supset\dots\supset\mathcal{E}^{(r)}=0
\end{equation*}
with $\mathcal{E}^{(i)}/\mathcal{E}^{(i+1)}\simeq\mathcal{L}$ for all $i$ and $\mathcal{E}$ is semifinite. By the previous result (1), we find that $\mathcal{C}(\mathcal{E}_2\otimes\mathcal{L})$ is a full subcategory of $\mathcal{C}^{EN}(X)$. 
We show that it is closed under taking subobjects in $\mathcal{C}^{EN}(X)$. Take  $\mathcal{E}\in\Ob(\mathcal{C}^{EN}(X))$, $\mathcal{F}\in\Ob(\mathcal{C}(\mathcal{E}_2\otimes\mathcal{L}))$ and an inclusion $\mathcal{E}\hookrightarrow\mathcal{F}$ in $\mathcal{C}^{EN}(X)$. Since the semi-simplification of an object of $\mathcal{C}^{EN}(X)$ is unique, up to isomorphisms, and $\mathcal{E}$ must be obtained by successive extensions of the simple objects which are subquotients of $\mathcal{F}$. Therefore, $\mathcal{E}$ also belongs to $\mathcal{C}(\mathcal{E}_2\otimes\mathcal{L})$.

(3) Take an indecomposable bundle $\mathcal{E}$ in $\mathcal{C}^{EN}(X)$. We must show that 
it is of the form $\mathcal{E}_r\otimes\mathcal{L}$ for some $r>0$ and some torsion line bundle $\mathcal{L}$. By definition, there exists a torsion line bundle $\mathcal{L}$ and an inclusion $\mathcal{L}\hookrightarrow\mathcal{E}$ in $\Coh (X)$(Remark \ref{rem:finite bundle on elliptic curve}). Then we have 
$\Gamma(\mathcal{E}\otimes\mathcal{L}^{\vee})=\Hom_{\calO_X}(\mathcal{L}, \mathcal{E})\neq 0$ 
and by Theorem \ref{thm:atiyah}(1), $\mathcal{E}\otimes\mathcal{L}^{\vee}$ must be isomorphic to $\mathcal{E}_r$, where $r=\rank{\mathcal{E}}$, whence $\mathcal{E}\simeq\mathcal{E}_r\otimes\mathcal{L}$.
\end{proof}

\begin{lem}\label{lem:essential extension}
Let $\mathcal{L}$ be a torsion line bundle on $X$. Fix a surjection of sheaves $f:\mathcal{E}_r\otimes\mathcal{L}\to\mathcal{L}$ with $\Ker(f)=\mathcal{E}_{r-1}\otimes\mathcal{L}$. Then $f$ defines an essential extension in
$\mathcal{S}(X)$. 
\end{lem}

\begin{proof}
Take a nonzero proper subobject $g:\mathcal{F}\hookrightarrow\mathcal{E}_r\otimes\mathcal{L}$ in $\mathcal{S}(X)$. We must show that  $f\circ g$ can not be an epimorphism. Since the functor $\rm{id}\otimes\mathcal{L}$ induces an equivalence of $\mathcal{S}(X)$ onto itself, we can reduce the lemma to the case where $\mathcal{L}=\mathcal{O}_X$. Assume that $\mathcal{F}\stackrel{g}{\hookrightarrow}\mathcal{E}_r\stackrel{f}{\twoheadrightarrow}\mathcal{O}_X$ is surjective. Since $\mathcal{C}^{\uni}(X)$ is closed under taking subobjects in $\mathcal{S}(X)$, $\mathcal{F}$ is a unipotent bundle. Thus, Lemma \ref{lem:nil on elliptic} (2) implies that $\mathcal{F}$ is of the form $\oplus_{i=1}^{t}\mathcal{E}_{r_i}$. Since we have assumed that $\mathcal{F}$ is nonzero, we find that $1\le\Dim\Gamma(\mathcal{F})$. 
On the other hand, since $\Gamma(\quad)$ is left exact, we have  $\Dim\Gamma(\mathcal{F})\le\Dim\Gamma(\mathcal{E}_r)=1$. For the final equality, we use Theorem \ref{thm:atiyah}(5). Thus $\Gamma(\mathcal{F})=k$ and $\mathcal{F}$ is indecomposable.  Namely, we have $\mathcal{F}\simeq\mathcal{E}_s$ for some $s<r$. Now we may identify $\mathcal{F}$ with $\mathcal{E}_s$. By taking dual of $\mathcal{E}_s=\mathcal{F}\stackrel{g}{\hookrightarrow}\mathcal{E}_r\stackrel{f}{\twoheadrightarrow}\mathcal{O}_X$, we have the following commutative diagram of two exact sequences
\begin{equation*}
\begin{CD}
@. @. \Ker(g^{\vee}) @. @. @.\\
@. @. @VVV @. @. @.\\
0@>>>\mathcal{O}_X@>f^{\vee}>>\mathcal{E}_r^{\vee}@>>>\Cok(f^{\vee})@>>>0\\
@.         @| @Vg^{\vee}VV @VVV @.\\
0@>>>\mathcal{O}_X@>(f\circ g)^{\vee}>>\mathcal{E}_s^{\vee}@>>>\Cok((f\circ g)^{\vee})@>>>0,
\end{CD}
\end{equation*} 
where $g^{\vee}$ is surjective. Since $\Gamma(\quad)$ is left exact and $\Gamma(\mathcal{O}_X)=\Gamma(\mathcal{E}_r^{\vee})=\Gamma(\mathcal{E}_s^{\vee})=k$ (Theorem \ref{thm:atiyah}(3)(5)), we find that $\Gamma(g^{\vee}):\Gamma(\mathcal{E}_r^{\vee})\to\Gamma(\mathcal{E}_s^{\vee})$ is an isomorphism. 

On the other hand, since $\rank\Ker(g^{\vee})=r-s>0$ and both $\mathcal{E}_r^{\vee}$ and $\mathcal{E}_s^{\vee}$ are unipotent bundles, the kernel of the surjection $g^{\vee}:\mathcal{E}_r^{\vee}\to\mathcal{E}_s^{\vee}$ is again a nonzero unipotent bundle.  Then we have an exact sequence
\begin{equation}\label{eq:lem 3.8-1}
0\to\Gamma(\Ker(g^{\vee}))\hookrightarrow \Gamma(\mathcal{E}_r^{\vee})\stackrel{\Gamma(g^{\vee})}{\to}\Gamma(\mathcal{E}_s^{\vee}).
\end{equation}
As $\Gamma(\Ker(g^{\vee}))\neq 0$ and $\Gamma(\mathcal{E}_r^{\vee})=k$, the inclusion $\Gamma(\Ker(g^{\vee}))\hookrightarrow \Gamma(\mathcal{E}_r^{\vee})$ must be a bijective. Since $\Dim(\Gamma(\Ker(g^{\vee})))=\Dim\Gamma(\mathcal{E}_r^{\vee})=\Dim\Gamma
(\mathcal{E}_s^{\vee})=1$, the exactness of (\ref{eq:lem 3.8-1}) implies that $\Gamma(g^{\vee})=0$, which is a contradiction. Hence, $f:\mathcal{E}_r\to\mathcal{O}_X$ is an essential extension.
\end{proof}

\begin{lem}\label{lem:projective generator}

(1) In the category $\mathcal{C}^{\uni}(X)$, we have $\langle\mathcal{E}_r\rangle\subset\langle\mathcal{E}_{r+1}\rangle$ for any $r>0$, and $\Ob(\mathcal{C}^{\uni}(X))=\bigcup\Ob(\langle\mathcal{E}_r\rangle)$. Furthermore, $\mathcal{E}_r$ is a projective generator of $\langle\mathcal{E}_r\rangle$.

(2) Put $P_m\overset{{\rm def}}{=}\bigoplus_{\mathcal{L}^{\otimes m}=\mathcal{O}_X}\mathcal{L}$ for each $m>0$. In $\mathcal{C}^{N}(X)$, we have $\langle P_{m}\rangle\subset\langle P_{mn}\rangle$ for any $m,n>0$ and $\Ob(\mathcal{C}^{N}(X))=\bigcup\Ob(\langle P_m\rangle)$. Furthermore, $P_m$ is a projective generator of $\langle P_m\rangle$. 

(3) In $\mathcal{C}^{EN}(X)$, we have $\langle\mathcal{E}_r\otimes P_m\rangle\subset\langle\mathcal{E}_{r+1}\otimes P_{mn}\rangle$ for any $r, m, n>0$ and $\Ob(\mathcal{C}^{EN}(X))=\bigcup\Ob(\langle\mathcal{E}_r\otimes P_m\rangle)$. Furthermore, $\mathcal{E}_r\otimes P_m$ is a projective generator of $\langle\mathcal{E}_r\otimes P_m\rangle$. 
\end{lem}

\begin{proof}
(1) For each $r>0$, we have $\mathcal{E}_r\subset\mathcal{E}_{r+1}$, which implies that $\langle\mathcal{E}_r\rangle\subset\langle\mathcal{E}_{r+1}
\rangle$. Lemma \ref{lem:nil on elliptic} implies that $\Ob(\mathcal{C}^{\uni}(X))=\bigcup\Ob(\langle\mathcal{E}_r\rangle)$. We prove the final assertion. It suffices to show that the natural surjection $\mathcal{E}_r\to\mathcal{O}_X$ is an essential extension in  $\langle\mathcal{E}_r\rangle$ and that $\mathcal{E}_r$ is a projective object of $\langle\mathcal{E}_r\rangle$~\cite[Chapter 6, Lemma 6.5.6]{sz09}. The first assertion follows from Lemma \ref{lem:essential extension}. We will show that $\mathcal{E}_r$ is projective in $\langle\mathcal{E}_r\rangle$. Since $\mathcal{E}_r\to\mathcal{O}_X$ is an essential extension, it suffices to show that the equality
\begin{equation*}
\Dim\Hom_{\calO_X}(\mathcal{E}_r, \mathcal{E}_r)=r\times \Dim\End(\mathcal{O}_X)=r
\end{equation*}
holds~\cite[Chapter 6, Lemma 6.5.8]{sz09}. This follows from Theorem \ref{thm:atiyah}(3), (4) and (5). 

(2) From the definition of $P_m$, the first two assertion follows. We will show that  $P_m$ is a projective generator of $\langle P_m\rangle$. First note that any objects in $\langle P_m\rangle$ is a direct sum of $m$-torsion line bundles. Therefore, for any $\mathcal{E}\in\Ob(\langle P_m\rangle)$, we have
\begin{equation}\label{eq:lem 3.9-1}
\Dim_k\Hom_{\calO_X}(P_m, \mathcal{E})=\rank\mathcal{E}.
\end{equation}
Take any exact sequence in $\langle P_m\rangle$,
\begin{equation*}
0\to\mathcal{E}'\to\mathcal{E}\stackrel{f}{\to}\mathcal{E}''\to 0.
\end{equation*}
Since $\Hom_{\calO_X}(P_m, \quad)$ is left exact, we have
\begin{equation*}
\begin{aligned}
\Dim_k\im(\Hom_{\calO_X}(P_m, f))&=\Dim_k\Hom_{\calO_X}(P_m, \mathcal{E})-\Dim_k\Hom_{\calO_X}(P_m, \mathcal{E}')\\
&\stackrel{(\ref{eq:lem 3.9-1})}{=}\rank\mathcal{E}-\rank\mathcal{E}'=\rank\mathcal{E}''
\stackrel{(\ref{eq:lem 3.9-1})}{=}\Dim_k\Hom_{\calO_X}(P_m, \mathcal{E}''),
\end{aligned}
\end{equation*}
which tells us the left exact functor $\Hom_X(P_m, \quad)$ is indeed an exact functor of $\langle P_m\rangle$. Therefore, $P_m$ is a projective object in $\langle P_m\rangle$. Next, we show that $P_m$ is a generator of $\langle P_m\rangle$. Since we have already proved that $P_m$ is projective in $\langle P_m\rangle$, it suffices to show that, for any $0\neq\mathcal{E}\in\Ob(\langle P_m\rangle)$, we have $\Hom_X(P_m,\mathcal{E})\neq 0$. It follows from the equality (\ref{eq:lem 3.9-1}).

(3) For each $r, m, n>0$, we have $\mathcal{E}_r\otimes P_m\subset\mathcal{E}_{r+1}\otimes P_{mn}$, which implies that$\langle\mathcal{E}_r\otimes P_m\rangle\subset\langle\mathcal{E}_{r+1}\otimes P_{mn}\rangle$. Lemma \ref{lem:semi-finite on elliptic} implies that $\Ob(\mathcal{C}^{EN}(X))=\bigcup\Ob(\langle\mathcal{E}_r\otimes P_m\rangle)$. 
The last assertion follows from Lemma \ref{lem:essential extension} combined with the equality 
$\Dim\Hom_{\calO_X}(\mathcal{E}_r\otimes\mathcal{L}, \mathcal{E}_r\otimes\mathcal{L})=r$ 
for any torsion line bundle $\mathcal{L}$.
\end{proof}

Now we obtain the following:

\begin{thm}\label{thm:main1}
There is an isomorphism of affine $k$-group schemes
\begin{equation*}
\pi_1^{EN}(X, x)\simeq \pi_1^{N}(X, x)\times\pi_1^{\uni}(X, x).
\end{equation*}
\end{thm}

\begin{proof}[proof 1]
We calculate a Tannakian fundamental group by using the notion of projective generators~\cite[Section 6.1]{de90}. By Lemma \ref{lem:projective generator}(1),(2), we have
\begin{equation*}
\begin{aligned}
\pi_1^{\uni}(X, x)&=\varprojlim_{r}\Spec(\End(\mathcal{E}_r)^{\vee}),\\
\pi_1^{N}(X, x)&=\varprojlim_{m}\Spec(\End({P}_m)^{\vee}).
\end{aligned}
\end{equation*}
Combined with  Lemma \ref{lem:projective generator}(3), we have
\begin{equation*}
\begin{aligned}
\pi_1^{EN}(X, x)&=\varprojlim_{r,m}\Spec(\End(\mathcal{E}_r\otimes P_m)^{\vee})=\varprojlim_{r,m}\Spec(\End(P_m)^{\vee}\otimes \End(\mathcal{E}_r)^{\vee})\\
&=\varprojlim_{r,m}(\Spec(\End(P_m)^{\vee})\times\Spec(\End(\mathcal{E}_r)^{\vee}))\\
&=\pi_1^{N}(X, x)\times\pi_1^{\uni}(X, x),
\end{aligned}
\end{equation*}
where, for the second equality, we use Proposition \ref{prop:n-nil}(2). 
\end{proof}

\begin{proof}[proof 2]
We prove the theorem by using the notion of tensor products of Tannakian categories~\cite[Chapter 5]{de90}. Let $A,G$ be affine group schemes over $k$.  The natural inclusions $\Rep(A),\Rep(G)\subset\Rep(A\times G)$ make $\Rep(A\times G)$ the tensor product of $\Rep(A)$ and $\Rep(G)$: $\Rep(A\times G)=\Rep(A)\otimes\Rep(G)$~\cite[5.18]{de90}. First we note that every object in $\Rep(A\times G)$ can be embedded into an object of the form $\oplus_{i=1}^{m}V_i\otimes U_i$, where $V_i$ are objects in $\Rep(A)$ and $U_i$ are objects in $\Rep(G)$. Indeed, let $W$ be an object in $\Rep(A\times G)$. We endows $M\overset{{\rm def}}{=}W\otimes_k k[A\times G]$ with a comodule structure by $\rho_M\overset{{\rm def}}{=}{\rm id}_W\otimes\Delta$, where $\Delta$ is the coproduct of $k[A\times G]$. Then there is a non-canonical isomorphisim as comodules $M\simeq k[A\times B]^{\oplus n}$, where $n=\Dim W$. The equality $\rho_M\circ\rho_W=({\rm id}\otimes\Delta)\circ\rho_W=(\rho_W\otimes {\rm id})\circ \rho_W$ implies that $\rho_W:W\to M$ is a comodule morphism. It is injective because $({\rm id}\otimes \epsilon)\rho_W={\rm id}$, where $\epsilon:k[A\times G]\to k$ is a counit. Therefore, we find that there is an injective comudule morphism $W\to k[A\times G]^{\oplus n}=(k[A]\otimes k[G])^{\oplus n}$~\cite[Section 3.5 Lemma]{wa79}. Since any comodule can be written as a directed union of finite-dimensional subcomodules~\cite[Section 3.3 Theorem]{wa79}, there are finite-dimensional subrepresentations $V_i(i=1,\dots,m)$ of $k[A]$ and $U_i(i=1,\dots,m)$ of $k[G]$, respectively such that $W\subset\oplus_{i=1}^{m}V_i\otimes  U_i\subset (k[A]\otimes k[G])^{\oplus n}$. 

Now we apply this remark to our setting $A\overset{{\rm def}}{=}\pi_1^{N}(X,x), G\overset{{\rm def}}{=}\pi_1^{\uni}(X, x)$. The natural inclusions $\mathcal{C}^{N}(X),\mathcal{C}^{\uni}(X)\subset\mathcal{C}^{EN}(X)$ induce a tensor functor $\Phi:\Rep(\pi_1^{N}(X,x)\times\pi_1^{\uni}(X,x))=\mathcal{C}^{N}(X)\otimes\mathcal{C}^{\uni}(X)\to\mathcal{C}^{EN}(X)$. Then every object in $\mathcal{C}^{N}(X)\otimes\mathcal{C}^{\uni}(X)$ can be embedded into an object of the form $\oplus_{i=1}^{n}\mathcal{L}_i\otimes\mathcal{E}_i$, where $\mathcal{L}_i$ are objects in $\mathcal{C}^{N}(X)$ and $\mathcal{E}_i$ are objects in $\mathcal{C}^{\uni}(X)$. Thus Proposition \ref{prop:n-nil}(2) implies that $\Phi$ is fully faithful (cf.Proposition \ref{prop:appendix1}). On the other hand, Lemma \ref{lem:semi-finite on elliptic}(3) implies that $\Phi$ is essentially surjective. Hence $\Phi$ is an equivalence. Therefore, we have $\pi_1^{EN}(X, x)\simeq\pi_1^{N}(X,x)\times\pi_1^{\uni}(X,x)$.
\end{proof}


\section{SEMIFINITE BUNDLES ON AN ARBITRARY VARIETY}

In this section, we always assume that $k$ is of characteristic 0 and that $X\in\Ob(\Var)$ with $X(k)\neq\emptyset$. Fix a rational point $x\in X(k)$. If $X\in\Ob(\Cv)$, we denote by $g_X$, or simply $g$, the genus of $X$. 

\subsection{The difference between $\pi_1^{N}$ and $\pi_1^{EN}$ }

 In this subsection, we describe the quotient category of $\mathcal{C}^{N}(X)\hookrightarrow\mathcal{C}^{EN}(X)$ and give a proof of Theorem \ref{thm:int-main2}.
We adopt Esnault-Hai-Sun's method~\cite[Chapter 3]{ehs08}~(cf.~Remark \ref{rem:int-ehs}).
We denote by $\pi_1^{E}(X, x)$ the kernel of this projection $p^{N}$~(Remark \ref{rem:rem of semi-finite2}):
\begin{equation}\label{eq:def pi_1^E}
\pi_1^{E}(X, x)\overset{{\rm def}}{=}\Ker(p^{N}). 
\end{equation}
If $X$ is an elliptic curve, then we have already seen that $\pi_1^E(X,x)=\pi_1^{\uni}(X, x)$ (Theorem \ref{thm:main1}). But, in general case, it is larger than $\pi_1^{\uni}(X, x)$. 
We have an exact sequence of group schemes:
\begin{equation}\label{eq:main ex seq}
1\to\pi_1^{E}(X, x)\to\pi_1^{EN}(X, x)\stackrel{p^{N}}{\to}\pi_1^{N}(X, x)\to 1.
\end{equation}
We describe the representation category $\Rep(\pi_1^{E}(X, x))$ in terms of vector bundles.  
We first show an analogous lemma of \cite[Proposition 3.4]{ehs08}:

\begin{lem}\label{lem:etale-nil1}
Let $\pi:P\to X$ be a finite \'etale torsor over $X$ with $H^{0}(P, \mathcal{O}_P)=k$ and let $\mathcal{E}$ be a unipotent bundle on $X$. Then the natural functor $\pi^{*}$ induces a fully faithful functor of $\mathcal{C}^{\uni}(X)$ into $\mathcal{C}^{\uni}(P)$ and the essential image of $\pi^*$ is closed under taking subobjects in $\calC^{\uni}(P)$. 
\end{lem}

\begin{proof}
We will show that ${{\rm Res}}(\pi^*):\pi_1^{\uni}(P,p)\to\pi_1^{\uni}(X,x)$ is a faithfully flat morphism.
For this, we need to recall the dual of $\pi_1^{\uni}$~\cite[Chapter IV]{no82}. Let $R_X$ be the ring of functions of $\pi_1^{\uni}(X, x)$, namely $\pi_1^{\uni}(X, x)=\Spec R_X$, and put $A_X\overset{{\rm def}}{=}R_X^{\vee}=\Hom_{k}(R_X, k)$. Let $\mathfrak{m}_X$ be the kernel of the natural projection $A_X=R_X^{\vee}\twoheadrightarrow k$. Nori proved the following:
\begin{equation*}
A_X=\varprojlim_{n}A_X/\mathfrak{m}_X^{n};\quad\mathfrak{m}_X/\mathfrak{m}_X^2=H^1(X, \mathcal{O}_X)^{\vee}.
\end{equation*}
Now we claim that the induced $H^1(\mathcal{O}_X)\to H^1(\mathcal{O}_P)$. Indeed, let $\calE\in H^1(\calO_X)={\rm Ext}^1(\calO_X,\calO_X)$. Then we have:
\begin{equation*}
\begin{aligned}
\Hom_{\calO_P}(\pi^*\calE,\calO_P^{\oplus 2})&\simeq\Hom_{\calO_X}(\calE,\pi_*\calO_P)^{\oplus 2}\\
&\simeq\Hom_{\calO_X}(\calE,\calO_X)\otimes\Hom_{\calO_X}(\calO_X,\pi_*\calO_P)^{\oplus 2}\\
&\simeq\Hom_{\calO_X}(\calE,\calO_X)\otimes H^0(P,\calO_P)^{\oplus 2}\\
&\simeq\Hom_{\calO_X}(\calE,\calO_X^{\oplus 2}),
\end{aligned}
\end{equation*}
where the second isomorphism comes from Proposition \ref{prop:n-nil}~(2) and the last one comes from the assumption $H^0(P,\calO_P)=k$. This implies that $\calE$ is trivial if and only if $\pi^*\calE$ is trivial.
Thus the map $H^1(\mathcal{O}_X)\to H^1(\mathcal{O}_P)$ is injective and the map $\mathfrak{m}_P/\mathfrak{m}_P^2\to\mathfrak{m}_X/\mathfrak{m}_X^2$ is surjective, whence so is $A_P\to A_X$~\cite[Chapter IV, Lemma 10]{no82}. Thus, the homomorphism $R_X\to R_P$ must be injective, which implies that $R_P$ is faithfully flat over $R_X$~\cite[Section 14.1]{wa79}.
\end{proof}

\begin{rem}\label{rem:genus}
 Nori proved that if $X\in\Ob(\Cv)$ with $g=\Dim H^1(\calO_X)$, then the dual $A_X$ is isomorphic to $k\langle\langle x_1,\dots, x_g\rangle\rangle$.
\end{rem}

\begin{lem}\label{lem:etale-nil2}
Let $\pi:P\to X$ be a finite \'etale torsor with $H^{0}(P, \mathcal{O}_P)=k$ and let $\mathcal{F}$ be a unipotent bundle on $P$. Then $\pi_{*}\mathcal{F}$ is a semifinite bundle on $X$.
\end{lem}

\begin{proof}
For the proof, we use induction on $\rank \mathcal{F}$. If $\rank\mathcal{F}=1$, then $\mathcal{F}=\mathcal{O}_P$. Thus $\pi_{*}\mathcal{F}=\pi_{*}\mathcal{O}_P$ is a finite bundle. If $\rank\mathcal{F}>1$, then there is an exact sequence
\begin{equation*}
0\to\mathcal{F}'\to\mathcal{F}\to\mathcal{O}_P\to 0.
\end{equation*}
Since $\pi$ is a finite morphism, the sequence
\begin{equation}\label{eq:lem 4.2-1}
0\to\pi_{*}\mathcal{F}'\to\pi_{*}\mathcal{F}\to\pi_{*}\mathcal{O}_P\to 0
\end{equation}
is also exact. By the induction hypothesis, $\pi_{*}\mathcal{F}'$ is semifinite and $\pi_{*}\mathcal{O}_P$ is finite. Thus the exactness of (\ref{eq:lem 4.2-1}) implies that $\pi_{*}\mathcal{F}$ is also semifinite.
\end{proof}

Let $S$ be a finitely generated full tensor subcategory of $\mathcal{C}^{N}(X)$. Then by Section 2.2, there is a finite $\pi(X,S,x)$-torsor $\pi_S:(X_S,x_S)\to (X,x)$. Recall that each $X_S$ is geometrically-connected and, since now $\pi(X,S,x)$ is \'etale, $X_S$ is geometrically-reduced. 
Now we construct the quotient category of $\mathcal{C}^{N}(X)\hookrightarrow\mathcal{C}^{EN}(X)$ in the similar way to  \cite[Definition 3.5]{ehs08}:

\begin{definition}
The category $\mathcal{C}^{E}(X, x)$ has for objects pairs $(X_S, \mathcal{E})$ where $S\subset\mathcal{C}^{N}(X)$ is a finitely generated full tensor subcategory, $\mathcal{E}\in\Ob(\mathcal{C}^{\uni}(X_S))$, and for morphisms 
\begin{equation*}
\Hom((X_S, \mathcal{E}), (X_T, \mathcal{F}))\overset{{\rm def}}{=}\varinjlim_{U\supset S\cup T}\Hom_{\calO_{X_U}}(\pi_{UT}^{*}\mathcal{E}, \pi_{US}^{*}\mathcal{F})
\end{equation*}
where $T\cup S$ is the full tensor subcategory of $\mathcal{C}^{N}(X)$ generated by $T$ and $S$. The composition rule is defined as follows. For any $\phi_{ij}\in\Hom((X_{S_i}, \mathcal{E}_{i}), (X_{S_j}, \mathcal{E}_{j}))$, $(i, j)=(1, 2), (2, 3)$, we put 
\begin{equation*}
\phi_{2, 3}\circ\phi_{1, 2}\overset{{\rm def}}{=}\pi_{S_1\cup S_2\cup S_3, S_2\cup S_3}^{*}\phi_{23}\circ\pi_{S_1\cup S_2\cup S_3, S_1\cup S_2}^{*}\phi_{12}\in \Hom((X_{S_1}, \mathcal{E}_1), (X_{S_3}, \mathcal{E}_3)).
\end{equation*} 
\end{definition}

\begin{prop}\label{prop:quotient category}
The category $\mathcal{C}^{E}(X, x)$ has a natural structure of a neutral Tannakian category over $k$.
\end{prop}

This proposition is a formal consequence of Proposition \ref{prop:appendix2} in Appendix. Here we only provide a description of the additive and tensor structures and a neutral fiber functor. For details, see Proposition \ref{prop:appendix2} in Appendix:
\begin{equation*}
\begin{aligned}
&\text{(additive structure)}& (X_S, \mathcal{E})\oplus(X_T, \mathcal{F})&\overset{{\rm def}}{=}(X_{S\cup T}, \pi_{S\cup T, S}^{*}\mathcal{E}\oplus\pi_{S\cup T, T}^{*}\mathcal{F});\\
&\text{(tensor structure)}& (X_S, \mathcal{E})\otimes(X_T, \mathcal{F})&\overset{{\rm def}}{=}(X_{S\cup T}, \pi_{S\cup T, S}^{*}\mathcal{E}\otimes\pi_{S\cup T, T}^{*}\mathcal{F});\\
&\text{(unit object)}& \mathbb{I}&\overset{{\rm def}}{=}(X, \mathcal{O}_X);\\
&\text{(neutral fiber functor)}&
\omega_{E}:\mathcal{C}^{E}(X, x)&\to\Vecf;\,(X_S, \mathcal{E})\mapsto x_S^{*}\mathcal{E}.
\end{aligned}
\end{equation*}
For a finitely generated full tensor subcategory $S\subset\mathcal{C}^{EN}(X)$, we put $S^{N}\overset{{\rm def}}{=}S\cap\mathcal{C}^{N}(X)$. This is a finitely generated full tensor subcategory of $\mathcal{C}^{N}(X)$. We define a functor $q:\mathcal{C}^{EN}(X)\to\mathcal{C}^{E}(X, x)$ by
\begin{equation}\label{eq:functor q}
q(\mathcal{E})\overset{{\rm def}}{=}(X_{\langle\mathcal{E}\rangle_\otimes^{N}}, \pi_{\langle\mathcal{E}\rangle_\otimes^{N}}^{*}\mathcal{E}).
\end{equation}
Then $q$ is a tensor functor compatible with the fiber functors $\omega_x$ and $\omega_{E}$.

\begin{thm}\label{thm:main2}
The category $\mathcal{C}^{E}(X, x)$ with the functor $q$ is the quotient category of $\mathcal{C}^{N}(X)\hookrightarrow\mathcal{C}^{EN}(X)$ with respect to the neutral fiber functor $\omega_x$ of $\mathcal{C}^{N}(X)$. The Tannakian fundamental group of $\mathcal{C}^{E}(X, x)$ with fiber funtcor $\omega_{E}$ is isomorphic to $\pi_{1}^{E}(X, x)$~(\ref{eq:def pi_1^E}).
\end{thm}

Let $\pi_1(\mathcal{C}^{E}(X,x),\omega_E)$ be the Tannakian fundamental group of $\mathcal{C}^{E}(X, x)$ with respect to the fiber functor $\omega_{E}$. The tensor functor $q$ induces a morphism of affine group schemes $q^{*}:\pi_1(\mathcal{C}^{E}(X,x),\omega_E)\to\pi_1^{EN}(X, x)$. If $\mathcal{E}$ is finite, then by definition of $q$ (\ref{eq:functor q}), we find that $q(\mathcal{E})$ is trivial in $\mathcal{C}^{E}(X, x)$~(cf.~Section 2.2 (2.6)). This implies that $q^{*}$ factors through $\pi_1^{E}(X, x)$ and we obtain the morphism of group schemes
\begin{equation}\label{eq:hom into pi_1^E}
\pi_1(\mathcal{C}^{E}(X,x),\omega_E)\to\pi_1^{E}(X, x)
\end{equation}
such that the following diagram commutes:
\begin{equation}\label{eq:comm diag}
\begin{CD}
0@>>>\pi_1(\mathcal{C}^{E}(X,x),\omega_E)@>q^{*}>>\pi_1^{EN}(X,x)@>p^{N}>>\pi_1^{N}(X,x)@>>>0\\
@. @V(\ref{eq:hom into pi_1^E})VV @| @| @.\\
0@>>>\pi_1^{E}(X,x)@>>>\pi_1^{EN}(X,x)@>p^{N}>>\pi_1^{N}(X,x)@>>>0.
\end{CD}
\end{equation} 
By definition of $\pi_1^{E}(X, x)$, the bottom row in (\ref{eq:comm diag}) is exact. It suffices to show that the top row in (\ref{eq:comm diag}) is exact. We need the following criterion of the exactness of a sequence of Tannakian fundamental groups~\cite[Appendix A, Theorem A.1]{ehs08}:

\begin{lem}\label{lem:exactness}
Let $L\stackrel{q}{\to}G\stackrel{p}{\to}A$ be a sequence of homomorphisms of affine group schemes over a field $k$. It induces a sequence of functors $\Rep(A)\stackrel{p^{*}}{\to}\Rep(G)\stackrel{q^{*}}{\to}\Rep(L)$. 

(i) The map $p$ is faithfully flat if and only if $p^{*}$ is fully faithful and $\Rep(A)$ is closed under taking subobjects in $\Rep(G)$. 

(ii) The map $q$ is a closed immersion if and only if any object of $\Rep(L)$ is isomorphic to a subquotient of $q^{*}(V)$ for some $V\in\Rep(G)$. 

(iii) Assume that $q$ is a closed immersion and that $p$ is faithfully flat. Then the sequence $L\stackrel{q}{\to}G\stackrel{p}{\to}A$ is exact if and only if the following conditions are fulfilled:

(a) For any object $V\in\Rep(G)$, $q^{*}(V)$ is trivial in $\Rep(L)$ if and only if $V\simeq p^{*}(U)$ for some $U\in\Rep(A)$. 

(b) Let $W_0$ be the maximal trivial subobject of $q^{*}(V)$ in $\Rep(L)$. Then there exists $V_0\subset V$ in $\Rep(G)$, such that $q^{*}(V_0)\simeq W_0$. 

(c) Any $W$ in $\Rep(L)$ is a quotient of $q^{*}(V)$ for some $V\in\Rep(G)$.
\end{lem}

\begin{proof}[Proof of Theorem \ref{thm:main2}](cf.~\cite[Theorem 3.8]{ehs08})
We have already seen that $\mathcal{C}^{N}(X)$ is a full subcategory of $\mathcal{C}^{EN}(X)$ which is closed under taking subobjects (Remark \ref{rem:rem of semi-finite2}). Thus by Lemma \ref{lem:exactness}(i), $p^{N}$ is faithfully flat. 
We show that $q^{*}$ is a closed immersion. Take any object $(X_S, \mathcal{E})$ in $\mathcal{C}^{E}(X, x)$. By Lemma \ref{lem:etale-nil2}, ${\pi_S}_*\calE$ belongs to $\mathcal{C}^{EN}(X)$. In fact, more precisely, ${\pi_S}_*\calE$ belongs to $\overline{S}$~(cf. Remark \ref{rem:rem of semi-finite1}; See also Remark \ref{rem:rem of main2-1}). We claim $(X_S,\mathcal{E})$ is a quotient of $q({\pi_S}_*\calE)=(X_S,\pi_S^{*}{\pi_S}_*\calE)$. This follows from the surjectivity of the adjunction map 
$\pi_{S}^{*}{\pi_{S}}_{*}\mathcal{E}\twoheadrightarrow\mathcal{E}$, which is valid because $\pi_S$ is finite \'etale. 
Thus, by Lemma \ref{lem:exactness}(ii), $q^{*}$ is a closed immersion. 

Now we will show that the sequence $\pi_1(\mathcal{C}^{E}(X,x),\omega_E)\stackrel{q^{*}}{\to}\pi_1^{EN}(X, x)\stackrel{p^{N}}{\to}\pi_1^{N}(X, x)$ is exact. We have to prove that it satisfies the conditions (a), (b) and (c) in Lemma \ref{lem:exactness}(iii). We have already showed that the condition (c) is fulfilled and the equivalence (\ref{eq:section 2 equiv})
 in Section 2.2 implies that the condition (a) is satisfied. It remains to be proven that the condition (b) is fulfilled. 
Let $\mathcal{E}$ be an object in $\mathcal{C}^{EN}(X)$. We put $S\overset{{\rm def}}{=}\langle\mathcal{E}\rangle_{\otimes}^{N}\subset\mathcal{C}^{N}(X)$. Let $\mathcal{F}_0$ be the maximal trivial subobject of $\pi_S^{*}\mathcal{E}$ in $\mathcal{C}^{\uni}(X_S)$. Let $r\overset{{\rm def}}{=}\rank\mathcal{F}_0$. Note that $r=\Dim H^{0}(X_S, \pi_S^{*}\mathcal{E})$. The projection formula implies that 
\begin{equation*}
\begin{aligned}
H^{0}(X_S, \pi_S^{*}\mathcal{E})&=\Hom_{\calO_{X_S}}(\mathcal{O}_{X_S}, \pi_S^{*}\mathcal{E})=\Hom_{\calO_X}(\mathcal{O}_X, {\pi_S}_{*}\pi_S^{*}\mathcal{E})\\
&=\Hom_{\calO_X}(\mathcal{O}_X, {\pi_S}_{*}(\mathcal{O}_{X_S}\otimes
\pi_S^{*}\mathcal{E}))\simeq\Hom_{\calO_X}(\mathcal{O}_X, ({\pi_S}_{*}\mathcal{O}_{X_S})\otimes\mathcal{E})\\
&\simeq\Hom_{\calO_X}(({\pi_S}_{*}\mathcal{O}_{X_S})^{\vee}, \mathcal{E})=\oplus_{i=1}^{r}k. \phi_i,
\end{aligned}
\end{equation*}
where $\{\phi_i\}$ is a basis of $\Hom_{\calO_X}(({\pi_S}_{*}\mathcal{O}_{X_S})^{\vee}, \mathcal{E})$. Let $\mathcal{E}_0$ be the image of the morphism $\oplus_{i=1}^{r}\phi_i:\oplus_{i=1}^{r}({\pi_S}_{*}\mathcal{O}_{X_S})^{\vee}\to\mathcal{E}$. Since $({\pi_S}_{*}\mathcal{O}_{X_S})^{\vee}$ is finite, $\mathcal{E}_0$ is a finite bundle in $\langle\mathcal{E}\rangle_{\otimes}\cap\mathcal{C}^{N}(X)=S$ and $\pi_S^{*}\mathcal{E}_0$ is trivial. We must show that $\pi_S^{*}\mathcal{E}_0\simeq\mathcal{F}_0$. It suffices to prove that $\rank\pi_S^{*}\mathcal{E}_0=r$.  Again, by the projection formula, we have
\begin{equation*}
\begin{aligned}
k^{\oplus r}=H^0(X_S, \pi_S^{*}\mathcal{E})&=\Hom_{\calO_X}(({\pi_S}_{*}\mathcal{O}_{X_S})^{\vee}, \mathcal{E})=\Hom_{\calO_X}(({\pi_S}_{*}\mathcal{O}_{X_S})^{\vee}, \mathcal{E}_0)\\
&=\Hom_{\calO_X}(\mathcal{O}_X, {\pi_S}_{*}\mathcal{O}_{X_S}\otimes\mathcal{E}_0)\simeq\Hom_{\calO_X}(\mathcal{O}_X,{\pi_S}_{*}(\mathcal{O}_{X_S}\otimes\pi_S^{*}\mathcal{E}_0))\\
&\simeq\Hom_{\calO_{X_S}}(\mathcal{O}_{X_S},\pi_S^{*}\mathcal{E}_0)\simeq H^0(X_S, \pi_S^{*}\mathcal{E}_0).
\end{aligned}
\end{equation*}
Therefore, we have $\rank\pi_S^{*}\mathcal{E}_0=r$, which completes the proof. 
\end{proof}

\begin{rem}\label{rem:rem of main2-1}
The same proof of Theorem \ref{thm:main2} implies the following. Namely, for any finitely generated tensor subcategory $S\subset\calC^N(X)$, the sequence of functors
\begin{equation*}
S\hookrightarrow\overline{S}\xrightarrow{\pi_S^*}\calC^{\uni}(X_S)
\end{equation*}
induces the following exact sequence:
\begin{equation*}
1\to\pi_1^{\uni}(X_S,x_S)\to\pi(X,\overline{S},x)\to\pi(X,S,x)\to 1.
\end{equation*}
Here, $\pi(X,\overline{S},x)$ is the quotient of $\pi_1^{EN}(X,x)$ corresponding to the subcategory $\overline{S}\subset\calC^{EN}(X)$~(Remark \ref{rem:rem of semi-finite1}).
\end{rem}

\begin{cor}\label{cor:cor of main2}
Let $\Lambda=(X_S, x_S)$ be the  pro-system of finite torsors over $X$ associated with $x$. Then for each $S$, there exists a faithfully flat morphism $u_S:\pi_1^{E}(X,x)\to\pi_1^{\uni}(X_S, x_S)$. If $S=\langle\mathcal{O}_X\rangle$, we simply write $u$ for $u_{\langle\calO_X\rangle}$. Furthermore, there exists an isomorphism of group schemes 
\begin{equation*}
\pi_1^{E}(X, x)\simeq\varprojlim_{\Lambda}\pi_1^{\uni}(X_S, x_S).
\end{equation*}
\end{cor}

\begin{proof}
It follows from Theorem \ref{thm:main2},  combined with Lemma \ref{lem:etale-nil1}~(cf.~Proposition \ref{prop:appendix2}~(4)):
\end{proof}

\begin{rem}\label{rem:rem of main2-2}
(i) Let $k$ be an algebraically closed field of characteristic 0 and $X$ an elliptic curve over $k$. In this case, Theorem \ref{thm:main1} and Corollary \ref{cor:cor of main2} imply that each multiplication map $n:X\to X$ induces an equivalence of categories of $\mathcal{C}^{\uni}(X)$ into itself. This fact have been already remarked by Atiyah~\cite[Some Aplications, Theorem 15(i)]{at57}.

(ii) If $X\in\Ob(\Cv)$ with $X$ smooth and $g\ge 2$, then $\pi_1^{E}(X, x)$ is strictly larger than $\pi_1^{\uni}(X, x)$, namely the projection $u:\pi_1^{E}(X,x)\twoheadrightarrow\pi_1^{\uni}(X,x)$  cannot be an isomorphism (cf.Corollary \ref{cor:cor of main2}) because the Riemann-Hurwitz formula implies that $g_{X_S}>g_X$ for $S\neq \langle\mathcal{O}_X\rangle$~(Remark \ref{rem:genus}). Furthermore, the extension (\ref{eq:main ex seq}) is non-trivial, i.e., $\pi_1^{EN}(X,x)\not\simeq\pi_1^{N}(X,x)\times\pi_1^{E}(X,x)$ because the unipotent part of $\mathcal{C}^{EN}(X)$ is exactly $\mathcal{C}^{\uni}(X)$. 
\end{rem}


\subsection{Action of $\pi_1^{N}$ on $\mathcal{C}^{E}$}

In this subsection, we give a proof of Theorem \ref{thm:int-main3}.
Here, we always assume that $k$ is an algebraically closed field of characteristic 0.
We first recall Hochschild and Mostow's theorem for pro-algebraic groups~\cite[Theorem 3.2 and Theorem 3.3]{hm69}:

\begin{thm}
(Hochschild and Mostow)
Take an exact sequence of affine group schemes over an algebraically closed field of characteristic 0
\begin{equation*}
0\to U\to G\to R\to 0
\end{equation*}
with $U$ unipotent and $R$ reductive. Then the projection $G\to R$ always has a section, up to unique conjugation by an element of $U$.  
\end{thm}

By applying this theorem to our setting (\ref{eq:main ex seq}), we have:
\begin{equation*}
\pi_1^{EN}(X, x)\simeq\pi_1^{N}(X, x)\ltimes\pi_1^{E}(X, x).
\end{equation*}
The map $\pi_1^{E}(X, x)\times\pi_1^{EN}(X, x)\to\pi_1^{E}(X, x); (n, g)\mapsto g^{-1}ng$ defines an (infinite dimensional) representation $\rho_{E}:\pi_1^{EN}(X, x)\to\Aut(k[\pi_1^{E}(X, x)])$. 
Fix a setion $t$ of the projection $p^{N}:\pi_1^{EN}(X, x)\to\pi_1^{N}(X, x)$. We denote by $\rho_{E}^{t}$ the image of $\rho_{E}$ by the induced tensor functor $t^{*}:{{\rm Ind}}(\Rep(\pi_1^{EN}(X, x)))\to{{\rm Ind}}(\Rep(\pi_1^{N}(X, x)))$, where ${{\rm Ind}}$ stands for  the \textit{ind-category}. Note that for an $k$-affine group scheme $G$, ${{\rm Ind}}(\Rep(G))$ is equivalent to the category of (not necessarily finite dimensional) left linear represenations of $G$ over $k$. 

\begin{thm}\label{thm:main3}
Under the above notation, if $X\in\Ob(\Cv)$ with $X$ smooth and $g\ge 2$, then the representation 
\begin{equation}\label{eq:main3}
\rho_{E}^{t}:\pi_1^{N}(X, x)\to\Aut(k[\pi_1^{E}(X, x)])
\end{equation}
is faithful.
\end{thm}

\begin{proof}
We first reduce the theorem to the faithfulness of a sub-representation
\begin{equation}\label{eq:main3-1}
\rho_{S}:\pi(X, S, x)\to\Aut(k[\pi_1^{\uni}(X_S, x_S)])
\end{equation}
for each finitely generated full tensor category $S\subset\mathcal{C}^{N}(X)$. Indeed, by Remark \ref{rem:rem of main2-1}, we have the following commutative diagram of exact sequences
\begin{equation}\label{eq:main3-2}
\begin{CD}
0@>>>\pi_1^{E}(X,x)@>q^{*}>>\pi_1^{EN}(X,x)@>p^{N}>>\pi_1^{N}(X,x)@>>>0\\
@. @Vu_SVV @Vg_SVV @Vr_SVV @.\\
0@>>>\pi_1^{\uni}(X_S,x_S)@>>>\pi(X,\overline{S},x)@>p_S>>\pi(X,S,x)@>>>0.
\end{CD}
\end{equation} 
Note that all the vertical arrows in (\ref{eq:main3-2}) are faithfully flat morphisms. We show that $t\circ g_S$ factors through $\pi(X, S, x)$. Put $N_S\overset{{\rm def}}{=}\Ker(r_S)$. Fix a section $t'$ of $p_S$. Since $g_S$ is surjective, $g_S\circ t(\pi_1^{N}(X, x))$ is a reductive subgroup of $\pi(X, \overline{S}, x)$ (cf.~\cite[Section 3]{hm69}). Thus, by \cite[Theorem 3.3]{hm69}, there exists an element $u\in\pi_1^{\uni}(X_S, x_S)$ such that $u (g_S\circ t(\pi_1^{N}(X, x))) u^{-1}\subset t'(\pi(X, S, x))$, which implies that $g_S\circ t(N_S)=1$. Thus, there exists a unique section $t_S$ of $p_S$ such that $g_S\circ t=t_S\circ r_S$. Denote by $\rho_S$ a representation $\pi(X, S, x)\to\Aut(k[\pi_1^{\uni}(X_S, x_S)])$ induced by $t_S$. We find that the representation $\rho_{E}^{t}$ is an inductive limit of $\rho_S$ where $S$ runs over all finitely generated full tensor categories of $\mathcal{C}^{N}(X)$. Therefore, it suffices to show that each $\rho_S$ is a faithful representation. 
Let $S\subset\mathcal{C}^{N}(X)$ be a finitely generated full tensor subcategory. We will show that the representation $\rho_S:\pi(X, S, x)\to\Aut(k[\pi_1^{\uni}(X_S, x_S)])$ is faithful. Assume that $\rho_S$ is not faithful. Then there exists a full tensor subcategory $T\subsetneq S$ so that $\pi(X, T, x)\simeq\pi(X, S, x)/\Ker(\rho_S)$. Then $\pi_{ST}:X_S\to X_T$ is a torsor with finite structure group $G_{ST}=\Ker(\pi(X, S, x)\twoheadrightarrow\pi(X, T, x))$. Put $\mathcal{E}\overset{{\rm def}}{=}{\pi_{ST}}_{*}\mathcal{O}_{X_S}$, which is a finite bundle on $X_T$. We have $G_{ST}=\pi(X_T, \langle\mathcal{E}\rangle, x_T)$. We have the following exact sequence
\begin{equation}\label{eq:main3-3}
1\to\pi_1^{\uni}(X_S, x_S)\to\pi(X_T, \overline{\langle\mathcal{E}\rangle} ,x_T)\to\pi(X_T,  \langle\mathcal{E}\rangle, x_T)\to1.
\end{equation}
Since $\pi(X_T, \langle\mathcal{E}\rangle ,x_T)=\Ker(\rho_S)$, the extension (\ref{eq:main3-3}) is trivial and
\begin{equation*}
\pi(X_T, \overline{\langle\mathcal{E}\rangle} ,x_T)=\pi(X_T, \langle\mathcal{E}\rangle ,x_T)\times\pi_1^{\uni}(X_S, x_S).
\end{equation*}
By taking the maximal unipotent quotient of both sides, we find that there is an equivalence of categories $\mathcal{C}^{\uni}(X_T)\simeq\mathcal{C}^{\uni}(X_S)$. However, since $T\subsetneq S$, the degree of $\pi_{ST}:X_S\to X_T$ is strictly larger than one and the genus of $X_S$ is different from the one of $X_T$, which is a contradiction~(Remark \ref{rem:genus} and Remark \ref{rem:rem of main2-2}(ii)). Therefore, the kernel of $\rho_S$ is trivial. This completes the proof.
\end{proof}


\section{BASIC PROPERTIES of $\pi_1^{EN}$}


\subsection{The base change property and the K\"unneth formula}

In this subsection, we see that $\pi_1^{EN}$ satisfies the base-change property for algebraic extensions of a base field and
the K\"unneth formula. Let $k$ be a field of characteristic 0 and fix an algebraic closure $\overline{k}$ of $k$. Let $X\in\Ob(\Var)$ and put $X_{\overline{k}}\overset{{\rm def}}{=}X\times_k \overline{k}$ and $\overline{x}\overset{{\rm def}}{=}x\times_k \overline{k}$. We denote by $p$ the natural projection $X_{\overline{k}}\to X$. We first show that $\pi_1^{E}$ has a base change property:

\begin{lem}\label{lem:base change for pi1^E}
There exists an isomorphism of group schemes $\pi_1^{E}(X_{\overline{k}}, \overline{x})\stackrel{\simeq}{\to}\pi_1^{E}(X, x)\times_k \overline{k}$.
\end{lem}

\begin{proof}
Let $\Lambda=(X_S, x_S)_S$ be the pro-system of finite torsors over $X$ associated with the fiber functor $\omega_x:\mathcal{C}^{N}(X)\to\Vecf$. By the base change property of $\pi_1^{N}$\cite[Chapter II, Proposition 5]{no82}, we find that $\Lambda_{\overline{k}}=((X_S)_{\overline{k}},\overline{x_S})$ is the pro-system of \'etale torsors over $X_{\overline{k}}$ associated with $\omega_{\overline{x}}:\mathcal{C}^{N}(X_{\overline{k}})\to\rm{Vec}_{\overline{k}}$. Therefore, we have
\begin{equation*}
\pi_1^{E}(X_{\overline{k}},\overline{k})=\varprojlim_{\Lambda_{\overline{k}}}\pi_1^{\uni}((X_S)_{\overline{k}},\overline{x_S})\simeq\varprojlim_{\Lambda}\pi_1^{\uni}(X_S,x_S)\times_k \overline{k}=\pi_1^{E}(X,x)\times_k \overline{k},
\end{equation*}
where for the second isomorphism, we use the base change property of $\pi_1^{\uni}$\cite[Chapter IV, Proposition 9]{no82}.
\end{proof}

\begin{prop}\label{prop:base change for pi1^EN}
There exists an isomorphism of group schemes $\pi_1^{EN}(X_{\overline{k}}, \overline{x})\stackrel{\simeq}{\to}\pi_1^{EN}(X, x)\times_k \overline{k}$.
\end{prop}

\begin{proof}
We define the category $\mathcal{D}$ as the full subcategory of $\mathcal{C}^{EN}(X_{\overline{k}})$ consisting of objects of $\mathcal{C}^{EN}(X_{\overline{k}})$ which can be embedded into $p^{*}\mathcal{E}$ for some $\mathcal{E}\in\Ob(\mathcal{C}^{EN}(X))$. Then $\mathcal{D}$ is a sub Tannakian category of $\mathcal{C}^{EN}(X_{\overline{k}})$ which is closed under taking subobjects and its Tannakian fundamental group is isomorphic to $\pi_1^{EN}(X,x)\times_k \overline{k}$~\cite[Proof of Proposition 3.1]{ms02}. We prove this. For ease of natation, we put $G\overset{{\rm def}}{=}\pi_1^{EN}(X,x)$ and $G'\overset{{\rm def}}{=}G\times_k \overline{k}$. Let $P\to X$ be the $G$-torsor associated with the inverse functor $F$ of $\mathcal{C}^{EN}(X)\stackrel{\simeq}{\to}\Rep(G)$(cf.Section 2.2). Then the base change $P_{\overline{k}}\to X_{\overline{k}}$ is a $G'$-torsor on $X_{\overline{k}}$. We define a funcor $F':{\rm Rep}_{\overline{k}}(G')\to\Qcoh(X_{\overline{k}})$ by putting $F'(V)\overset{{\rm def}}{=}\mathcal{O}_{P_{\overline{k}}}\otimes_{\overline{k}}V/G'$. If $V\in\Ob({\rm Rep}_{\overline{k}}(G'))$ is defined over $k$, namely $V=W\otimes_k \overline{k}$ for some $W\in\Ob(\Rep(G))$, then we have
\begin{equation*}
\begin{aligned}
\Hom_{{\rm Rep}_{\overline{k}}(G')}(1,V)&=V^{G'}=W^{G}\otimes_k \overline{k}\\
&=\Hom_{\Rep(G)}(1,W)\otimes\overline{k}
\simeq\Hom_{\calO_X}(\mathcal{O}_X,F(W))\otimes\overline{k}\\
&=\Hom_{\calO_{X_{\overline{k}}}}(\mathcal{O}_{X_{\overline{k}}},F(W)\otimes\overline{k})
=\Hom_{\calO_{X_{\overline{k}}}}(\mathcal{O}_{X_{\overline{k}}},F'(V)).
\end{aligned}
\end{equation*}
However, every object $V\in\Ob({\rm Rep}_{\overline{k}}(G'))$ can be embedded into an object of the form $W\otimes \overline{k}$ for some $W\in\Ob(\Rep(G))$. Therefore, by Proposition \ref{prop:appendix1}, we find that $F'$ is fully faithful. The equality $F(W)\otimes\overline{k}=F(W\otimes\overline{k})$ for $W\in\Ob(\Rep(G))$ implies that the essential image of $F'$ is $\mathcal{D}$. Therefore we have $\mathcal{D}\simeq{\rm Rep}_{\overline{k}}(G')$.

Hence, we obtain a faithfully flat morphism $\pi_1^{EN}(X_{\overline{k}},\overline{k})\to\pi_1^{EN}(X,x)\times_k \overline{k}$. However, this projection can be inserted in the following commutative diagram of exact sequences
\begin{equation*}
\begin{CD}
0@>>>\pi_1^{E}(X_{\overline{k}},\overline{x})@>>>\pi_1^{EN}(X_{\overline{k}}, \overline{x})@>>>\pi_1^{N}(X_{\overline{k}}, \overline{x})@>>>0\\
@. @V\simeq VV @VVV @V\simeq VV @.\\
0@>>>\pi_1^{E}(X, x)\times_k \overline{k}@>>>\pi_1^{EN}(X,x)\times_k \overline{k}               @>>>\pi_1^{N}(X,x)\times_k \overline{k}@>>>0.
\end{CD}
\end{equation*} 
The left vertical arrow is isomorphic by Lemma \ref{lem:base change for pi1^E} and the right vertical arrow is isomorphic by the base change property of $\pi_1^{N}$\cite[Chapetr II, Proposition 5]{no82}. Thus so is the middle one, which completes the proof.
\end{proof}

Next, we will show the K\"unneth formula for $\pi_1^{EN}$. We need again the one for $\pi_1^{E}$:

\begin{lem}\label{lem:kunneth for pi1^E}
Let $X, Y\in\Ob(\Var)$ with $x\in X(k), y\in Y(k)$. Then there exists an isomorphism of  group schemes $\pi_1^{E}(X\times Y, (x, y))\stackrel{\simeq}{\to}\pi_1^{E}(X, x)\times\pi_1^{E}(Y, y)$.
\end{lem}

\begin{proof}
Let $\Lambda_X=(X_S,x_S), \Lambda_Y=(Y_T,y_T)$ be the pro-systems of \'etale torsors associated with the fiber functors $\omega_x:\mathcal{C}^{N}(X)\to\Vecf$ and $\omega_y:\mathcal{C}^{N}(Y)\to\Vecf$, respectively. The K\"unneth formula for $\pi_1^{N}$~(cf.~Remark \ref{rem:kunneth}) implies that $\Lambda\overset{{\rm def}}{=}(X_S\times Y_T,(x_S,y_T))$ is the pro-system associated with $\omega_{(x,y)}:\mathcal{C}^{N}(X\times Y)\to\Vecf$. Thus we have
\begin{equation*}
\begin{aligned}
\pi_1^{E}(X\times Y,(x,y))&=\Spec\varinjlim_{S,T}k[\pi_1^{\uni}(X_S\times Y_T,(x_S,y_T))]\\
&=\Spec\varinjlim_{S,T}k[\pi_1^{\uni}(X_S,x_S)]\otimes k[\pi_1^{\uni}(Y_T,y_T)]\\
&=\Spec(\varinjlim_{S}k[\pi_1^{\uni}(X_S,x_S)])\otimes(\varinjlim_{T}k[\pi_1^{\uni}(Y_T,y_T)])\\
&=\pi_1^{E}(X,x)\times\pi_1^{E}(Y,y),
\end{aligned}
\end{equation*}
where the second equality follows from the K\"unneth formula for $\pi_1^{\uni}$~\cite[Chapter IV, Lemma 8]{no82}. This finishes the proof.
\end{proof}

\begin{rem}\label{rem:kunneth}
From the base change property of $\pi_1^N$~\cite[Chapter II, Proposition 5]{no82}, in the characteristic 0 case, the K\"unneth formula for $\pi_1^N$ is a consequence of the one for the \'etale fundamental group $\pi_1$, which is due to Grothendieck~\cite{gr71}. 
\end{rem}

\begin{prop}\label{prop:kunneth for pi1^EN}
Let $X, Y\in\Ob(\Var)$ with $x\in X(k), y\in Y(k)$. Then there exists an isomorphism of group schemes $\pi_1^{EN}(X\times Y, (x, y))\stackrel{\simeq}{\to}\pi_1^{EN}(X, x)\times\pi_1^{EN}(Y, y)$.
\end{prop}

\begin{proof}
Let ${\rm pr}_1:X\times Y\to X, {\rm pr}_2:X\times Y\to Y$. These maps induce a faithfully flat morphism $\pi_1^{EN}(X\times Y,(x,y))\to\pi_1^{EN}(X,x)\times\pi_1^{EN}(Y,y)$. This morphism can be inserted in the following commutative diagram:
\begin{footnotesize}
\begin{equation*}
\begin{xy}
\xymatrix{
0\ar[r]&\pi_1^E(X\times Y,(x,y))\ar[r]\ar[d]^{\simeq}&\pi_1^{EN}(X\times Y, (x,y))\ar[r]\ar[d]&\pi_1^{N}(X\times Y, (x,y))\ar[r]\ar[d]^{\simeq}& 0\\
0\ar[r]&\pi_1^{E}(X,x)\times\pi_1^{E}(Y,y)\ar[r]&\pi_1^{EN}(X,x)\times\pi_1^{EN}(Y,y)\ar[r]&\pi_1^{N}(X,x)
\times\pi_1^{N}(Y,y)\ar[r]& 0,
}
\end{xy}
\end{equation*}
\end{footnotesize}where the left and right vertical arrows are isomorphic by Lemma \ref{lem:kunneth for pi1^E} and the K\"unneth formula for $\pi_1^{N}$~(Remark \ref{rem:kunneth}), respectively. Therefore the middle one is also isomorphic, which completes the proof.
\end{proof}


\subsection{On the homotopy sequence for $\pi_1^{EN}$}

In this subsection, we discuss on the homotopy sequence for $\pi_1^{EN}$. Let $k$ be a field of characteristic 0 and let $X,S\in\Ob(\Var)$ together with rational points $x\in X(k)$ and $s\in S(k)$. Let $f:X\to S$ be a separable (proper) morphism of $X$ to $S$  with geometrically connected fibers and with $f(x)=s$. Consider the following diagram:
\begin{equation*}
\begin{xy}
\xymatrix{\ar@{}[rd]|{\square}
X_s\ar[r]^{t}\ar[d]_g& X\ar[d]^f\\
\Spec k\ar[r]^s& S.
}
\end{xy}
\end{equation*}
This induces the following complex of affine group schemes
\begin{equation}\label{eq:hmtp seq}
\pi_1^{EN}(X_s,x)\overset{t_*}{\to}\pi_1^{EN}(X,x)\overset{f_*}{\to}\pi_1^{EN}(S,s)\to 1.
\end{equation}

We first show that the sequence (\ref{eq:hmtp seq}) is exact at $\pi_1^{EN}(S,s)$:

\begin{lem}\label{lem:right exactness}
Under the above notation, the homomorphism $f_*:\pi_1^{EN}(X,x)\to\pi_1^{EN}(S,s)$ is a faithfully flat morphism.
\end{lem}

\begin{proof}
From the base change property for $\pi_1^{EN}$~(Proposition \ref{prop:base change for pi1^EN}), we may assume that $k$ is an algebraically closed field. Since $f_*\calO_X=\calO_S$, the projection formula implies the functor $f^*:\calC^{EN}(S)\to\calC^{EN}(X)$ is fully faithful. It remains to be proven that the essential image of $f^*$ is closed under taking subobjects in $\calC^{EN}(X)$. Take any object  $\calE\in\Ob(\calC^{EN}(S))$ and any subbundle $\calE'\subset f^*\calE$ in $\calC^{EN}(X)$. We must show that there exists a semifinite bundle $\calF$ on $S$ such that $\calE'\simeq f^*\calF$. Then the same argument in the proof of \cite[Lemma 8.1~(b)]{la11} implies that $f_*\calE'$ is locally free on $S$ and $\calE'\simeq f^*f_*\calE'$. Therefore, it suffices to show that $f_*\calE'$ is semifinite. If $\calE'$ is finite, then so is $f_*\calE'$~\cite{zh13}, so we may assume that $\calE'$ is not finite. Then there exists a nonzero finite subbundle $\calE''\subset\calE'$ such that $\rank\calE''<\rank\calE'$. Then the quotient $\calE'''\overset{{\rm def}}{=}\calE'/\calE''$ also satisfies these properties, i.e., $f_*\calE'''$ is locally free and $f^*f_*\calE'''\simeq\calE'''$. Therefore, by induction, $f_*\calE'''$ is semifinite on $S$. Furthermore, the sequence 
\begin{equation*}
0\to f_*\calE''\to f_*\calE'\to f_*\calE'''\to 0
\end{equation*}
is exact because of the left exactness of $f_*$ and the equality:
$\rank f_*\calE'''=\rank f_*\calE'-\rank f_*\calE''$.
This implies that $f_*\calE'$ is also semifinite. This finishes the proof.
\end{proof}

Therefore, the exactness of the sequence (\ref{eq:hmtp seq}) is equivalent to the exactness at $\pi_1^{EN}(X,x)$. Now we seek to say further this condition in terms of vector bundles as \cite[Theorem 3.1]{zh13}. For this, we need the universality of $\pi_1^{EN}$. The following lemma is inspired by the related work due to Borne-Vistoli~(cf.~Remark \ref{rem:int2}):

\begin{lem}\label{lem:universality of pi_1^EN}~(cf.~Remark \ref{rem:universality of pi_1^N})
Let $(X_x^{EN},x^{EN})$ be the universal $\pi_1^{EN}(X,x)$-torsor with:
\begin{equation*}
F(X_x^{EN}):\Rep(\pi_1^{EN}(X,x))\overset{\omega_x^{-1}}{\simeq}\calC^{EN}(X)\subset\Qcoh(X).
\end{equation*}
Then for any \textit{locally unipotent} algebraic affine group scheme $G$ over $k$, the natural map
\begin{equation*}
\Hom(\pi_1^{EN}(X,x),G)\to{{\rm Tors}}((X,x),G);
\,\phi\mapsto (X_x^{EN}\times G)/\pi_1^{EN}(X,x)
\end{equation*}
is bijective. Here ${{\rm Tors}}((X,x),G)$ stands for the set of isomorphism classes of $G$-torsors $P\to X$ together with a rational point $p$ above $x$. 
\end{lem}

 Here an affine algebraic group $G$ is said to be \textit{locally unipotent} if the connected component $G^0$ of the identity element of $G$ is unipotent.

\begin{proof}
There exists a bijection between the set ${\rm Tors}((X,x),G)$ and the set of $k$-linear exact $\otimes$-functors $F$ of $\Rep(G)$ into $\Qcoh(X)$ with $x^*\circ F$ the forgetful functor~\cite[Proposition 2.9]{no76}\cite[Chapter 1, Proposition 2.9]{no82}. Thus it suffices to show that for any $k$-linear exact $\otimes$-functor $F:\Rep(G)\to\Qcoh(X)$ factors through the inclusion $\calC^{EN}(X)\hookrightarrow\Qcoh(X)$. If $G^0=1$, then $G$ is finite, whence $F(\Rep(G))\subset\calC^{N}(X)\subset\calC^{EN}(X)$~(\cite[Proposition 3.8]{no76}\cite[Chapter I, Proposition 3.8]{no82}). Therefore, we may assume that $G^0\neq 1$. Let $0\neq V\in\Ob(\Rep(G))$. Then since $G^0$ is unipotent, $V^{G^0}\neq 0$. Since $F$ is exact, we have the following exact sequence in $\Qcoh(X)$:
\begin{equation*}
0\to F(V^{G^0})\to F(V)\to F(V/V^{G^0})\to 0.
\end{equation*}
Since $V^{G^0}$ comes from the representation of the finite group scheme $G/G^{0}$, $F(V^{G^0})$ is a finite bundle on $X$. On the other hand, from the induction hypothesis, $F(V/V^{G^0})$ is a semifinite bundle on $X$. Therefore, the middle one $F(V)$ is also semifinite. This completes the proof. 
\end{proof}

From the universality of $\pi_1^{EN}$, together with the Tannakian interpretation, 
we obtain the analogous criterion for the exactness of the homotopy sequence (\ref{eq:hmtp seq}) as \cite[Theorem 3.1]{zh13}. Before stating it, we recall the notion of \textit{saturated  torsors}~\cite[Definition 1]{zh13}:
Let $G$ be a locally unipotent algebraic group over $k$. Then a $G$-tosror $(P,p)\to (X,x)$ is said to be \textit{saturated} if the corresponding homomorphim $\pi_1^{EN}(X,x)\to G$~(Lemma \ref{lem:universality of pi_1^EN}) is surjective.

\begin{thm}(cf. \cite[Theorem 3.1]{zh13})\label{thm:hmtp seq} The homotopy sequence (\ref{eq:hmtp seq}) is exact if the following condition is satisfied:
For any saturated locally unipotent $G$-torsor $\pi:(P,p)\to (X,x)$ with $\pi_*\calO_P=\varinjlim_{\alpha}\calE_{\alpha}$ where $\calE_{\alpha}\in \calC^{EN}(X)$, we have

(a) the natural maps $s^*f_*\calE_{\alpha}\to g_*t^*\calE_{\alpha}$ are  isomorphisms, and

(b) the push-forward $f_*\calE_{\alpha}$ is a semifinite bundle on $S$ for any $\alpha$.
\end{thm}

\begin{proof}
We adopt the proof of  \cite[Theorem 3.1]{zh13}.
We first show that for each saturated $G$-torsor $\pi:(P,p)\to (X,x)$, there exists a saturated $G'$-torsor $\pi':(P',p')\to (S,s)$ together with a map 
$\theta:(G,P,p)\to f^*(G',P',p')$ 
with $(\pi'_*\calO_P')_s\overset{\simeq}{\to}(f_*\pi_*\calO_P)_s$.
Indeed, from the condition (a), for each $\alpha$, the adjunction map $f^*f_*\calE_{\alpha}\to\calE_{\alpha}$ is injective~(\cite[Theorem 3.1]{zh13}). Thus, so is $f^*f_*\pi_*\calO_P\to\pi_*\calO_P$. Let $G'$ be the quotient of $\pi_1^{EN}(S,s)$ corresponding to the Tannakian subcategory $\langle f_*\calE_{\alpha}\,|\,\alpha\rangle_{\otimes}\subset\calC^{EN}(S)$. Then the fully faithfulness of $f$ implies that $\langle f_*\calE_{\alpha}\,|\,\alpha\rangle_{\otimes}\overset{\simeq}{\to}\langle f^*f_{*}\calE_{\alpha}\,|\,\alpha\rangle_{\otimes}$, so $G'$ appears as a quotient of $G$. Let $H\overset{{\rm def}}{=}\Ker(G\twoheadrightarrow G')$ and let $M\subset k[G]=(k[G],\rho_{\rm reg})$ be the subrepresentation corresponding to $f^*f_*\pi_*\calO_P\subset\pi_*\calO_P$. Since $M$ comes from an object in ${\rm Ind}(\Rep(G'))$, the inclusion $M\hookrightarrow k[G]$ factors through $k[G]^{H}\hookrightarrow k[G]$, i.e., $M\subseteq k[G]^H\subset k[G]$. From the universality of $\pi_1^{EN}(S,s)$~(Lemma \ref{lem:universality of pi_1^EN}), the surjective map $\pi_1^{EN}(S,s)\twoheadrightarrow G'$ corresponds to some saturated $G'$-torsor $\pi':(P',p')\to (S,s)$ and we have a morphism of torsors, $\theta:(P,G,p)\to f^*(P',G',p')$.
Since $k[G']=k[G]^{H}$, the morphism $P\to f^*P'=P/H$ is faithfully flat. Therefore, we obtain an inclusion $f^*\pi'_*\calO_{P'}\subset\pi_*\calO_{P}$, which factors through $f^*f_*\pi_*\calO_P\subset\pi_*\calO_P$, whence $f^*\pi'_*\calO_{P'}\subset f^*f_*\pi_*\calO_P$. This implies that $k[G]^H\subseteq M$, whence $M=k[G]^H=k[G']$. This implies that 
$(\pi'_*\calO_{P'})_s\simeq (f_*\pi_*\calO_P)_s$. 

Finally we will show that the exactness of the sequence (\ref{eq:hmtp seq}). It suffices to show the exactness at $\pi_1^{EN}(X,x)$. Let $\pi_1^{EN}(X,x)\twoheadrightarrow G$ be any quotient algebraic group. Let $\pi:(P,G,p)\to (X,x)$ be the corresponding saturated torsor. Then from the above discussion, there exists a saturated torsor $\pi':(P',G',p')\to (S,s)$ together with a map $\theta:(P,G,p)\to f^*(P',G',p')$ inducing an isomorphism $(\pi'_*\calO_{P'})_s\simeq (f_*\pi_*\calO_P)_s$. Let
\begin{itemize}
\item $H\overset{{\rm def}}{=}\Ker(G\twoheadrightarrow G')$;
\item $N\overset{{\rm def}}{=}\im(\pi_1^{EN}(X_s,x)\to\pi_1^{EN}(X,x)\twoheadrightarrow G)$；
\item $N'\overset{{\rm def}}{=}\im(\Ker(f_*)\to\pi_1^{EN}(X,x)\twoheadrightarrow G)$.
\end{itemize}
Note that $N\subseteq N'\subseteq H(\subseteq G)$. We will show that $H=N=N'$. First, notice that the map $\theta |_{X_s}:f^*f_*\pi_*\calO_P|_{X_s}\hookrightarrow\calO_P|_{X_s}$ corresponds to the homomorphism of the representations $k[G]^H\hookrightarrow k[G]$ in ${{\rm Ind}}(\Rep(\pi_1^{EN}(X_s,x)))$. However, from the base properties of $\pi_*\calO_P$, we have $f^*f_*\pi_*\calO_P|_{X_s}\simeq \calO_{X_s}\otimes H^0(X_s,\pi_*\calO_P|_{X_s})$, 
so the subrepresentation $k[G]^H$ is the maximal trivial subrepresentation of $k[G]$ in the category ${{\rm Ind}}(\Rep(\pi_1^{EN}(X_s,x)))$, whence the equality $k[G]^H=k[G]^{N}$. This imples that $H=N$, whence $N=N'=H$.  This completes the proof.
\end{proof}

Now as an application, we prove that the homotopy sequence (\ref{eq:hmtp seq}) is exact for a projective \textit{Du Bois} family over a smooth projective curve $f:X\to S$. Here a projective flat family $f:X\to S$ over a proper scheme $S$ is said to be \textit{Du Bois} if for each closed point $s'\to S$, the fiber $f^{-1}(s')=X_{s'}$ is  Du Bois in the sense of \cite[Definition 3.2]{sc07}.

\begin{cor}\label{cor:hmtp seq}
Assume that $f:X\to S$ is a projective Du Bois family with $f(x)=s$. Then:

(1) The sequence (\ref{eq:hmtp seq}) is exact if any saturated locally unipotent $G$-torsor $\pi:(P,p)\to (X,x)$ with $\pi_*\calO_P=\varinjlim_{\alpha}\calE_{\alpha}$ satisfies the condition (b) in Theorem \ref{thm:hmtp seq}.

(2) If $S$ is a projective smooth curve over $k$, then the sequence (\ref{eq:hmtp seq}) is exact.
\end{cor}

\begin{lem}\label{lem:hmtp seq}
Assume that $f:X\to S$ is a projective Du Bois family with $f(x)=s$.

(1) For any finite bundle $\calE$ on $X$, the sheaf $R^if_*\calE$ is locally free and the natural map 
$s^*R^{i}f_*\calE\to H^i(X_s,\calE|_{X_s})$ 
is an isomorphism for any $i\ge 0$.

(2) Assume that $k$ is an algebraically closed field and that $S$ is a projective smooth curve over $k$. Then for any simple finite bundle $\calE$ on $X$, the dual sheaf $(R^1f_*\calE)^{\vee}$ is nef.
\end{lem}

\begin{proof}
(1) Since ${{\rm char}}k=0$, the category $\calC^{N}(X)$ is semi-simple~(Remark \ref{rem:semi-simple}), so we may assume that  $\calE=\pi_*\calO_P$ for some saturated finite torsor $\pi:P\to X$. By \cite[Theorem 12.11]{ha77}, it suffices to show that for any $i\ge 0$, the sheaf $R^if_*(\pi_*\calO_P)=R^i(f\circ\pi)_*\calO_P$ is locally free. However, since $f$ is Du Bois and $\pi$ is \'etale,  $f\circ\pi$ also defines a Du Bois family over $S$. Therefore, by \cite[Theorem 4.6]{du81}, the sheaves $R^i(f\circ\pi)_*\calO_P$ are locally free.

(2) By the Lefschetz principle, we may assume that $k=\mathbb{C}$. Since any quotient of a nef bundle is again nef, so it suffices to show that for any saturated finite torsor $\pi:(P,p)\to (X,x)$, $(R^1f_*(\pi_*\calO_P))^{\vee}=(R^1(f\circ\pi)_*\calO_P)^{\vee}$ is nef. It follows from \cite{pa15}.
\end{proof}

\begin{proof}[Proof of Corollary \ref{cor:hmtp seq}]
(1) From Proposition \ref{prop:base change for pi1^EN}, we may assume that $k$ is an algebraically closed field. By Theorem \ref{thm:hmtp seq}, it suffices to show that the condition (a) is satisfied. We will show that for any semifinite bundle $\calE$ on $X$, the natural map 
$s^*f_*\calE\to g_*t^*\calE$
is an isomorphism. Note that since the homotopy sequence for $\pi_1^N$ is exact, it is true for any finite bundle on $X$~\cite{zh13}. Thus we may assume that $\calE\in\calC^{EN}(X)\setminus\calC^N(X)$. Then there exists a finite subbundle $\calE'\subset\calE$ such that $\rank\calE'<\rank\calE$. We put $\calE''\overset{{\rm def}}{=}\calE/\calE'$.
Thus we are reduced to prove that the natural base-change map 
$s^*R^1f_*\calE'\to H^1(X_s,\calE'|_{X_s})$ 
is an isomorphism. Since $\calE'$ is finite, this follows from Lemma \ref{lem:hmtp seq}~(1).

(2) From Proposition \ref{prop:base change for pi1^EN}, we may assume that $k$ is an algebraically closed field. From the previous result and Theorem \ref{thm:hmtp seq}, it suffices to show that for any semifinite bundle $\calE$ on $X$, the push-forward sheaf $f_*\calE$ is again semifinite. This is true for any finite bundle $\calE$ on $X$~\cite{zh13}, so we may assume that $\calE\in\calC^{EN}(X)\setminus\calC^{N}(X)$. Then again there exists a simple finite subbundle $\calE'\subset\calE$ with $\rank\calE'<\rank\calE$. Put $\calE''\overset{{\rm def}}{=}\calE/\calE'$. By induction hypothesis, we may assume that $f_*\calE''$ is semifinite. Consider the following exact sequence
\begin{equation*}
0\to f_*\calE'\to f_*\calE\to f_*\calE''\overset{\alpha}{\to} R^1f_*\calE'. 
\end{equation*}
Since $\Ker\alpha\subset f_*\calE''$ and $f_*\calE''$ is locally free, the sheaf $\Ker\alpha$ is torsion-free. Hence $\Ker\alpha$ is locally free because $S$ is a smooth curve. Similarly for $\im\alpha$. The local-freeness of $f_*\calE''/\Ker\alpha\simeq\im\alpha$ implies that $\Ker\alpha$ is a subbundle of $f_*\calE''$. Note that  $\im\alpha$ is not necessarily a subbundle of $R^1f_*\calE'$, but there exists a subbundle $\overline{\im\alpha}$ of $R^1f_*\calE'$ containing $\im\alpha$ such that $\Deg(\im\alpha)=\Deg(\overline{\im\alpha})$ and $\rank(\im\alpha)=\rank(\overline{\im\alpha})$. The semistability of $f_*\calE''$ implies that $\mu(\Ker\alpha)\le\mu(f_*\calE'')=0$. On the other hand,
\begin{equation*}
\begin{aligned}
0&\ge\Deg(\Ker\alpha)=\Deg(f_*\calE'')-\Deg(\im\alpha)
=-\Deg(\im\alpha)\\
&=\Deg((\overline{\im\alpha})^{\vee})\ge\rank(\overline{\im\alpha})\cdot\mu_{{\rm min}}((R^1f_*\calE')^{\vee})\ge 0,
\end{aligned}
\end{equation*}
where for the last inequality, we use Lemma \ref{lem:hmtp seq}~(2).  Therefore, we have $\mu(\Ker\alpha)=0$, so again by the semistability of $f_*\calE''$, we find that $\Ker\alpha$ is semistable of degree 0. However, since the category $\calC^{EN}(S)$ is closed under taking subobjects in $\mathcal{S}(S)$, $\Ker\alpha$ belongs to $\calC^{EN}(S)$. Therefore, the exactness of the sequence
\begin{equation*}
0\to f_*\calE'\to f_*\calE\to \Ker\alpha\to 0
\end{equation*}
implies that the sheaf $f_*\calE$ is semifinite. This completes the proof.
\end{proof}

The same argument as above implies:

\begin{cor}
Assume that $f:X\to S$ is a projective Du Bois family over a smooth projective curve $S$ with $f(x)=s$. Then the following sequence is exact:
\begin{equation*}
\pi_1^{\uni}(X_s,x)\overset{t_*}{\to}\pi_1^{\uni}(X,x)\overset{f_*}{\to}\pi_1^{\uni}(S,s)\to 1.
\end{equation*}
\end{cor}


\appendix
\section{APPENDIX}

In this section, we prove some basic facts on tensor categories which are used in this paper. We first show the following:

\begin{prop}\label{prop:appendix1}
Let $\mathcal{C}$ be a $k$-linear abelian rigid tensor category and let $\mathcal{D}$ be a $k$-linear abelian tensor category.  Let $F:\mathcal{C}\to\mathcal{D}$ be an exact tensor functor. Assume that there exists a family $A\subset\Ob(\mathcal{C})$ satisfying the following conditions:

(a) The unit object $\mathbb{I}$ belongs to $A$.

(b) For any $V,W\in A$, the map $F:\Hom_{\mathcal{C}}(V,W)\to\Hom_{\mathcal{D}}(F(V),F(W))$ is bijective.

(c) For any $V\in\Ob(\mathcal{C})$, there exists an $W\in A$ such that $V\subset W$. 

Then the functor $F$ is fully faithful.
\end{prop}

\begin{proof}
Since $\mathcal{C}$ is rigid, it suffices to show that for any object $V\in\Ob(\mathcal{C})$, the map $F:\Hom_{\mathcal{C}}(\mathbb{I},V)\to\Hom_{\mathcal{D}}(\mathbb{I},F(V))$ is bijective. By assumption (c), there exists an exact sequence in $\mathcal{C}$
\begin{equation*}
0\to V\to W\to U
\end{equation*}
with $W,U\in A$. Hence, we obtain the following commutative diagram of exact sequences in $\Vecf$
\begin{equation*}
\begin{CD}
0@>>>\Hom_{\mathcal{C}}(\mathbb{I},V)@>>>\Hom_{\mathcal{C}}(\mathbb{I},W)@>>>\Hom_{\mathcal{C}}(\mathbb{I},U)\\
@. @VVV @V\simeq VV @V\simeq VV\\
0@>>>\Hom_{\mathcal{D}}(\mathbb{I},F(V))@>>>\Hom_{\mathcal{D}}(\mathbb{I},F(W))@>>>\Hom_{\mathcal{D}}(\mathbb{I},F(U))
\end{CD}
\end{equation*}
where the second and third vertical arrows are bijective from the assumption (a), (b). Thus so is the first one. 
\end{proof}

Next we prove the following proposition, which is used to prove Proposition \ref{prop:quotient category}:

\begin{prop}\label{prop:appendix2}
Let $I$ be a filtered category with initial object $0$, and $\{(\mathcal{C}_i, \omega_i), \pi_{ij}\}_{i\in I}$ be an inductive-system of neutral Tannakian categories over $k$ together with fibre functors $\omega_i:\mathcal{C}_i\to\Vecf$, where $\pi_{ij}:\mathcal{C}_i\to\mathcal{C}_j (i\le j)$ are $k$-linear exact tensor functors with $\omega_j\circ\pi_{ij}=\omega_i$.  We define a category $\mathcal{C}$ as follows. Objects are all pairs $(i, V)$ where $i\in I$ and $V\in\Ob(\mathcal{C}_i)$. For each objects $(i, V), (j, W)$, the set of morphisms is defined by 
\begin{equation*}
\Hom((i, V), (j, W))\overset{{\rm def}}{=}\varinjlim_{k\ge i, j}\Hom_{\mathcal{C}_k}(\pi_{ik}V, \pi_{jk}W).
\end{equation*}
The composition rules are defined as follows. Take any $\phi_{pq}=[f_{pq}]\in\Hom((p, V_p), (q, V_q))$,$(p, q)=(i, j), (j, k)$. Here we fix $f_{ij}\in\Hom_{\mathcal{C}_l}(\pi_{il}V_i,
\pi_{jl}V_j)$ and $f_{jk}\in\Hom_{\mathcal{C}_m}(\pi_{jm}V_j,\pi_{km}V_k)$, $l\ge i,j; m\ge j,k$. We put 
\begin{equation*}
\phi_{jk}\circ\phi_{ij}\overset{{\rm def}}{=}[\pi_{mn}f_{jk}\circ\pi_{ln}f_{ij}]\in \Hom((i, V_i), (k, V_k)).
\end{equation*} 
where $n\ge l, m$. For each $i\in I$, we define a functor $\pi_i:\mathcal{C}_i\to\mathcal{C}$ by  putting $\pi_i(V)\overset{{\rm def}}{=}(i, V)$. 

Assume the following conditions:

(a)for each $i\le j$, the functor $\pi_{ij}$ is fully faithful.

(b)for any finite set $S\subset \Ob(I)$, the subcategory $I_{\ge S}\overset{{\rm def}}{=}\{j\in I; j\ge i ( i\in S)\}$ has an initial object. We denote it by $\cup S$.
Then the following holds:

(1) The category $\mathcal{C}$ can be uniquely endowed with a structure of $k$-linear abelian category so that the functors $\pi_i (i\in I)$ are $k$-linear exact functors. 

(2) The category $\mathcal{C}$ can be uniquely endowed with a structure of tensor category so that the functors $\pi_i (i\in I)$ are tensor functors. Furthermore, under this tensor structure, the category $\mathcal{C}$ is rigid.

(3) Under the above structure (1) and (2), there is a unique neutral fibre functor $\omega$ of $\mathcal{C}$ such that $\omega\circ\pi_i=\omega_i$ for all $i\in I$. 

(4) Assume that for each $i\le j$, $\pi_{ij}:\mathcal{C}_i\hookrightarrow\mathcal{C}_j$ is closed under taking subobjects in $\mathcal{C}_j$. Then each $\pi_i$ induces a faithfully flat morphism $\pi(\mathcal{C},\omega)\to\pi(\mathcal{C}_i,\omega_i)$ and there is an isomorphism of $k$-affine group schemes
\begin{equation*}
\pi(\mathcal{C},\omega)\stackrel{\simeq}{\to}\varprojlim_{i\in I}\pi(\mathcal{C}_i,\omega_i),
\end{equation*}
where $\omega$ is a fiber functor of $\mathcal{C}$ in (3).
\end{prop}

\begin{proof}
(1) We first show that $\mathcal{C}$ have a finite direct product. For any objects $(i_1, V_1), (i_2, V_2)\in\Ob(\mathcal{C})$, we put $k\overset{{\rm def}}{=}i_1\cup i_2$. We show that $(k, \pi_{i_2, k}V_1\oplus \pi_{i_2, k}V_2)$ together with morphisms $[{\rm pr}_q:\pi_{i_1, k}V_1\oplus \pi_{i_2, k}V_2\to\pi_{i_q, k}V_q] (q=1, 2)$ is a direct product of these $(i_q, V_q)(q=1,2)$. Indeed, take two morphisms $g_q=[f_q]:(j, W)\to (i_q, V_q)$ with $f_q\in\Hom_{\mathcal{C}_m}(\pi_{jm}W, \pi_{i_q, m}V_q)$, $m\ge i_1, i_2, j$. Then there exist a unique morphism $f:\pi_{jm}W\to \pi_{i_1, m}V_1\oplus \pi_{i_2, m}V_2$ in $\mathcal{C}_m$ such that ${\rm pr_q}\circ f=f_q (q=1,2)$. Fix an index $n\ge m, k$. Then
 since $\pi_{mn}$ is fully faithful, $g\overset{{\rm def}}{=}[\pi_{mn}f]$ is a unique morphism of $(j, W)$ into $(k, \pi_{i_1, k}V_1\oplus \pi_{i_2, k}V_2))$ such that ${\rm pr}_q\circ g=g_q (q=1,2)$. Therefore, the categroy $\mathcal{C}$ has a additive structure.

Next we show that $\mathcal{C}$ has kernels and cokernels such that, for any morphism $f$, the natural map $\rm{Coim}(f)\to\rm{Im}(f)$ is isomorphism. It suffices to show that any morphism $f:V\to W$ in $\mathcal{C}_i$, $\pi_{ij}(\Ker(f))=\Ker(\pi_{ij}f)$ and $\pi_{ij}(\Cok(f))=\Cok(\pi_{ij}f)$ in $\mathcal{C}_j$ for any $j\ge i$. It follows from the exactness of $\pi_{ij}$.

Therefore, we have endowed the category $\mathcal{C}$ with a structure of a $k$-linear abelian category. Under this structure, the functors $\pi_i$ are $k$-linear exact functors.

Finally we show the uniqueness. Assume that $\mathcal{C}$ has been endowed with a structure of $k$-linear abelian category so that $\pi_i (i\in I)$ are $k$-linear exact functors. Take any two objects $(i_1, V_1), (i_2, V_2)\in\Ob(\mathcal{C})$.
By assumption, there is a direct product ${\rm pr}_q:(i_1, V_1)\oplus(i_2, V_2)\overset{{\rm def}}{=}(j, W)\to (i_q, V_q) (q=1,2)$ of these. By definition, these projections are represented by morphisms $p_q:\pi_{jm}W\to\pi_{i_q, m}V_q$ in $\mathcal{C}_m$ for some $m\ge j, i_1, i_2$. Since the functor $\pi_{m}:\mathcal{C}_m\to\mathcal{C}$ is additive, we find that $\pi_{jm}W\simeq \pi_{i_1, m}V_1\oplus\pi_{i_2, m}V_2$ in $\mathcal{C}_m$, which implies that $(i_1, V_1)\oplus(i_2, V_2)$ is isomorphic to $(m, \pi_{i_1, m}V_1\oplus \pi_{i_2, m}V_2)$ together with morphisms $[{\rm pr}_q:\pi_{i_1, m}V_1\oplus \pi_{i_2, m}V_2\to\pi_{i_q, m}V_q] (q=1, 2)$. Since $m\ge i_1, i_2$ factor through $i_1\cup i_2=:k$, we find that $(i_1, V_1)\oplus(i_2, V_2)$ is isomorphic to $(k, \pi_{i_1, k}V_1\oplus \pi_{i_2, k}V_2)$ together with morphisms $[{\rm pr}_q:\pi_{i_1, k}V_1\oplus \pi_{i_2, k}V_2\to\pi_{i_q, k}V_q] (q=1, 2)$.Therefore, the additive structure of $\mathcal{C}$ is uniquely determined. The structure of $k$-linear abelian category is uniquely determined by the condition that the functors $\pi_i$ are $k$-linear exact functors. 

(2)We define the functor $\otimes:\mathcal{C}\times\mathcal{C}\to\mathcal{C}$ by putting:
\begin{equation*}
(i, V)\otimes (j, W)\overset{{\rm def}}{=}(i\cup j, \pi_{i, i\cup j}V\otimes \pi_{j, i\cup j}W). 
\end{equation*}
We will show that it gives a tensor structure on $\mathcal{C}$. For each $i$, we denote by $\phi_i$ the associativity constraint $V\otimes (W\otimes U)\simeq (V\otimes W)\otimes U$ for $(\mathcal{C}_i,\otimes)$. We define an associativity constraint $\Phi$ for $(\mathcal{C},\otimes)$ as follows. Let $(i_k, V_k)\in\Ob(\mathcal{C})(k=1,2,3)$. We put $i_{12}\overset{{\rm def}}{=}i_1\cup i_2$ and so on. Now we define an isomorphim
\begin{equation*}
\begin{aligned}
\Phi_{123}:&(i_{123},\pi_{i_1,i_{123}}V_1\otimes(\pi_{i_2,i_{123}}V_2\otimes\pi_{i_3,i_{123}}V_3))=(i_1,V_1)\otimes ((i_2,V_2)\otimes (i_3,V_3))\\
&\to((i_1,V_1)\otimes (i_2,V_2))\otimes (i_3,V_3)=(i_{123},(\pi_{i_1,i_{123}}V_1\otimes\pi_{i_2,i_{123}}V_2)\otimes\pi_{i_3,i_{123}}V_3)
\end{aligned}
\end{equation*}
by putting $\Phi_{123}\overset{{\rm def}}{=}[\phi_{i_{123}}]$. We define $\Phi$ as the collection of $\Phi_{123}$. Since all $(\mathcal{C}_i,\otimes,\phi_i)$ satisfy the pentagon axiom, so is $(\mathcal{C},\otimes,\Phi)$. Similarly, the commutativity constraint $\Psi$ for $(\mathcal{C},\otimes)$ can be defined naturally and $(\mathcal{C},\otimes,\Phi,\Psi)$ satisfies the hexagon axiom. Let $(1,u)$ is a unit object of $\mathcal{C}_0$ with $u:1\stackrel{\simeq}{\to}1\otimes 1$. Then the isomorphim in $\mathcal{C}$
\begin{equation*}
[u]:(0,1)\to(0,1)\otimes(0,1)=(0,1\otimes 1)
\end{equation*}
makes $(0,1)\in\Ob(\mathcal{C})$ the unit object of $(\mathcal{C},\otimes)$. Therefore, the category $\mathcal{C}$ is a tensor category.

Under this tensor structure, the natural embedding $\pi_i:\mathcal{C}_i\to\mathcal{C}$ is a tensor functor. Furthermore, for each $(i,V)\in\Ob(\mathcal{C})$, the object $(i,V^{\vee})$ in $\mathcal{C}$ gives the dual object $(i,V)^{\vee}$. Indeed, it follows the following remark. Let $\mathcal{A},\mathcal{B}$ be tensor categories and $F:\mathcal{A}\to\mathcal{B}$ be a fully faithful tensor functor. Let $X,Y\in\Ob(\mathcal{A})$. Then $Y=X^{\vee}$ if and only if $F(Y)=F(X)^{\vee}$.

Finally, we show the uniqueness. Assume that $\mathcal{C}$ has a structure of a tensor category so that the functors $\pi_i:\mathcal{C}_i\to\mathcal{C}$ are tensor functors. Then for each $(i,V),(j,W)\in\Ob(\mathcal{C})$, we have
\begin{equation*}
\begin{aligned}
(i,V)\otimes (j,W)&=(i\cup j,\pi_{i,i\cup j}V)\otimes (i\cup j,\pi_{j,i\cup j}W)=\pi_{i\cup j}(\pi_{i,i\cup j}V)\otimes\pi_{i\cup j}(\pi_{j,i\cup j}W)\\
&=\pi_{i\cup j}(\pi_{i,i\cup j}V\otimes\pi_{j,i\cup j}W),
\end{aligned}
\end{equation*}
where the first equality is given by the canonical isomorphisms. Thus the tensor structure of $\mathcal{C}$ is uniquely determined by $\pi_i$ and the tensor structure of $\mathcal{C}_i$.

(3) We define a functor $\omega:\mathcal{C}\to\Vecf$ by putting $\omega(i,V)\overset{{\rm def}}{=}\omega_{i}(V)$. It is well-defined since $\omega_j\circ\pi_{ij}=\omega_i$. Furthermore, it is a tensor functor.
Note that a sequence
\begin{equation*}
0\to (i,V)\to (j,W)\to (k,U)\to 0
\end{equation*}
is exact in $\mathcal{C}$ if and only if the sequence in $\mathcal{C}_{i\cup j\cup k}$
\begin{equation*}
0\to\pi_{i,i\cup j\cup k}V\to\pi_{k,i\cup j\cup k}W\to\pi_{k,i\cup j\cup k}U\to 0
\end{equation*}
is exact in $\mathcal{C}_{i\cup j\cup k}$. Thus the functor $\omega$ is exact. Since $\mathcal{C}$ is $k$-linear abelian rigid tensor category, the faithfulness is also fulfilled~\cite[Corollary 2.10]{de90}. Thus it defines a fiber functor.

(4) For any $i\in I$, let $S_{ij}\subset\mathcal{C}_i(j\in J_i)$ be an inductive system of full  tensor subcategories which are closed under taking subobjects in $\mathcal{C}_i$ so that $\mathcal{C}_i=\cup_j S_{ij}$. Let $\pi(\mathcal{C}_i,\omega_i)\twoheadrightarrow\pi(S_{ij})$ be the corresponding quotient. Then we have $\pi(\mathcal{C}_i,\omega_i)=\varprojlim_{j\in J_i}\pi(S_{ij})$. Now we have $\mathcal{C}=\cup_{i}\pi_i(\mathcal{C}_i)=\cup_{i}\cup_{j}\pi_i(S_{ij})$. Since each $\pi_{i}(S_{ij})$ is closed under taking subobjects in $\mathcal{C}$ by assumption, we have 
\begin{equation*}
\pi(\mathcal{C},\omega)=\varprojlim_{i\in I,j\in J_i}\pi(S_{ij})=\varprojlim_{i\in I}\varprojlim_{j\in J_i}\pi(S_{ij})=\varprojlim_{i\in I}\pi(\mathcal{C}_i,\omega_i).
\end{equation*}
\end{proof}

\section*{ACKNOWLEDGMENT}

The author thanks Professor Takao Yamazaki for a lot of clarifying discussions and encouragements and for checking and correcting this paper in many times. The author thanks Professor Takuya Yamauchi for careful reading and giving him comments on this paper. The author thanks the referee for making time available to this paper and for giving him useful comments and suggestions.

\section*{FUNDING}

The author is supported by JSPS, Grant-in-Aid for Scientific Research for JSPS fellows (16J02171).

\end{document}